\documentclass[12pt]{amsart}
\usepackage{amsmath}
\usepackage{amssymb}
\usepackage{amscd}
\usepackage{graphicx,color}
\usepackage{spverbatim}
\usepackage{bm}
\newtheorem{theorem}{Theorem}[section]

\newtheorem{proposition}[theorem]{Proposition}

\theoremstyle{definition}
\newtheorem{definition}[theorem]{Definition}

\numberwithin{equation}{section}

\newtheorem{remark}[theorem]{Remark}

\title[A reduction theorem for good basic invariants]
{A reduction theorem for good basic invariants 
of finite complex reflection groups}
\author{Yukiko Konishi}
\address{ 
Department of Mathematics, 
College of Liberal Arts,
Tsuda University,
Toyko 187-8577, Japan 
}
\email{konishi@tsuda.ac.jp}
\author{Satoshi Minabe}
\address{Department of Mathematics, 
Tokyo Denki University, Tokyo 120-8551,  Japan}
\email{minabe@mail.dendai.ac.jp}
\begin{document}
\maketitle

\begin{abstract}
This is a sequel to our previous article \cite{KM2023}. We describe a certain reduction process
of Satake's good basic invariants. 
We show that 
if the largest degree $d_1$ of a finite complex reflection group $G$ is regular and if 
$\delta$ is a divisor of $d_1$, 
a set of good basic invariants of $G$ induces that of the reflection subquotient $G_{\delta}$.
We also show that
the potential vector field of a duality group $G$,
which gives the multiplication constants of the natural Saito structure on the orbit space, 
induces that of $G_{\delta}$. Several examples of this reduction process
are also presented.
\end{abstract}
\section{Introduction}
A  reflection on a complex vector space $V$ 
is a linear transformation of finite order
which fixes a hyperplane pointwise.
A finite subgroup $G$ of $\mathrm{GL}(V)$ is called 
a finite complex reflection group if it is generated by 
 reflections. It is well-known that such a group
is characterized by the property that the ring of polynomial 
invariants is freely generated, i.e. it is generated by a set of
algebraically independent homogeneous polynomials.
Such a set of generators is called a set of basic invariants.
The reader can find basic invariants in \cite{Mehta1988} for finite Coxeter groups,
\cite[\S 2.8]{LehrerTaylor} for imprimitive groups $G(m,p,n)$, 
\cite{ShephardTodd}\cite[\S 6.6]{LehrerTaylor} for primitive groups of rank two,
\cite[\S B.3]{OrlikTerao} and references therein for primitive groups of rank $\geq 3$.

As far as the authors know, there are two directions in pursuing distinguished sets of basic invariants.
The one direction, which we do not deal with, is the so-called canonical system.  See \cite{Iwasaki1997}, \cite{Iwasaki2002},
\cite{NakashimaTsujie2014}, \cite{NakashimaTeraoTsujie}, \cite{Talamini2018}.
The other direction is the flat generator system
proposed by Saito--Yano--Sekiguichi \cite{SaitoYanoSekiguchi} for finite Coxeter groups.
It is nowadays understood as a flat coordinate system of the structure of a Frobenius manifold
on the orbit space. See \cite{Saito1993}, \cite{Dubrovin1993-2}.   
Generalizations of these results to finite complex reflection groups 
are obtained by several groups of authors \cite{Arsie-Lorenzoni2016}, 
\cite{KatoManoSekiguchi2015}, \cite{KMS2018}.

Recently, Satake \cite{Satake2020}
introduced the notion of  good basic invariants for finite complex reflection groups.
For finite Coxeter groups, he showed that good basic invariants are flat 
in the sense of \cite{SaitoYanoSekiguchi}. Moreover he found a formula for 
multiplication constants of the Frobenius manifold structure on the orbit spaces.
In \cite{KM2023}, the authors showed that good basic invariants are flat coordinates 
of the natural Saito structures constructed in \cite{KMS2018} and gave a formula 
for the multiplication constants.

The aim of this paper is to describe a certain reduction process
of Satake's good basic invariants regarding  reflection subquotients
of finite complex reflection groups. The theory 
of the reflection subquotients was developed by Lehrer and Springer 
\cite{LehrerSpringer1}, \cite{LehrerSpringer2}.
Let $G$ be a finite complex reflection group.
If $E$ is the $\xi$-eigenspace of  a $\xi$-regular element of $G$, where $\xi$ is a primitive $\delta$-th
root of unity, then the subgroup
$N_E=\{s\in G\mid s(E)=E\}$ is a finite complex reflection group on $E$.
It is called a $\delta$-reflection subquotient of $G$ and denoted by $G_{\delta}$. See \S \ref{sec:definitions} for detail.
In this article, 
we show that a set of good basic invariants induces that of the reflection subquotient $G_{\delta}$
by a divisor $\delta$ of the largest degree $d_1$ of $G$ (Theorem \ref{thm:gbi}).
We also show that the potential vector field of a duality $G$,
which gives the multiplication constants of the natural Saito structure \cite{KM2023},
induces that of $G_{\delta}$ (Thoerem \ref{thm:pvf}).
The key to our arguments is the notion of an admissible triplet introduced by Satake \cite{Satake2020}.
An admissible triplet $(g,\zeta,q)$ consists of a $d_1$-th primitive root of unity
$\zeta$, 
a $\zeta$-regular element $g$ and a $\zeta$-regular eigenvector $q$ of $g$.
If a positive integer $\delta$ is a divisor of $d_1$ and $E$ is the $\zeta^{d_1/\delta}$-eigenspace of $g^{d_1/\delta}$, the corresponding reflection subquotient $G_{\delta}$
 has the same largest degree $d_1$ as $G$, and 
$(g|_E,\zeta,q)$ is an admissible triplet of $G_{\delta}$.
Hence there exists a sequence of reflection subquotients with the same $d_1$.
For example, for imprimitive groups, we have
\begin{equation}\label{monomial-sequence}
G(m,m,n+1)\stackrel{m}{\longrightarrow}G(m,1,n)\stackrel{k}{\longrightarrow} 
G\left(km,1,\frac{n}{k}\right)\quad (d_1=nm)~.
\end{equation}
Here in the first arrow, it is assumed that $n+1$ is not divisible by $m$
and in the second arrow, $k$ is a divisor of $n$.
As for primitive groups, we list two sequences  in Figures \ref{fig1} and \ref{fig2}.
\begin{figure}[t]
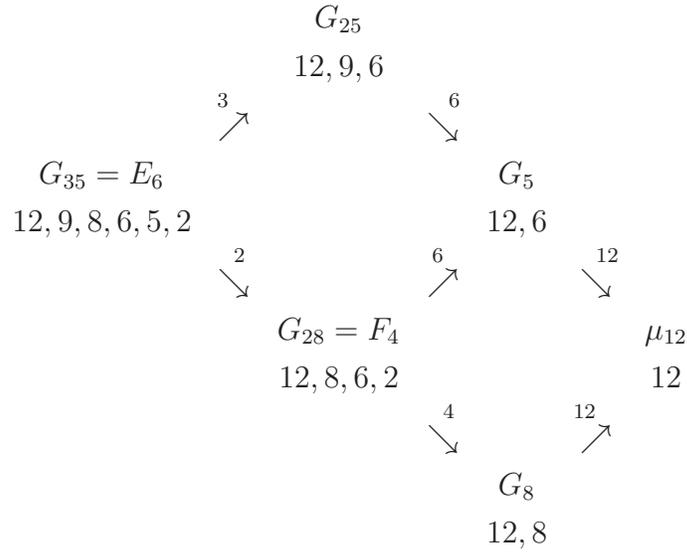

\begin{equation}\nonumber
\begin{array}{ccccccc}
&&G_{25}&&&&\\
&&12,9,6&&&&\\
&\stackrel{3\phantom{23}}{\nearrow}&&\stackrel{\phantom{23}6}{\searrow}&&&\\
G_{35}=E_6&&&&G_5&&\\
12,9,8,6,5,2&&&&12,6&&\\
&\stackrel{\phantom{3}2}\searrow&&\stackrel{6\phantom{2}}{\nearrow}&
&\stackrel{\phantom{23}12}{\searrow}&\\
&&G_{28}=F_4&&&&\mu_{12}\\
&&12,8,6,2&&&&12\\
&&&\stackrel{\phantom{2}4}{\searrow}&&\stackrel{\!\!\!12\phantom{2}}{\nearrow}&\\
&&&&G_8&&\\
&&&&12,8&&\\
\end{array}
\end{equation}
\caption{The sequence of reflection subquotients of $G_{35}=E_6$.
All those six groups  are duality groups with $d_1=12$.
The numbers written under each group are the degrees of that group. }
\label{fig1}
\end{figure}
\begin{figure}[t]
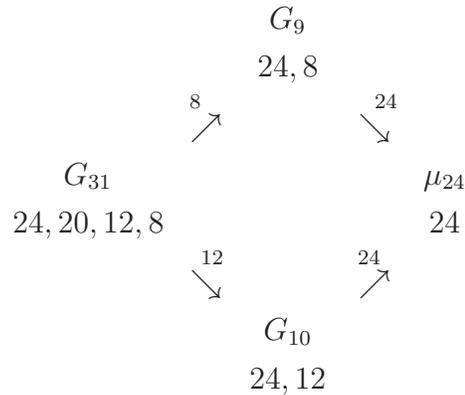

\begin{equation}\nonumber
\begin{array}{ccccccc}
&&G_{9}&&&&\\
&&24,8&&&&\\
&\stackrel{8\phantom{23}}{\nearrow}&&\stackrel{\phantom{23}24}{\searrow}&&&\\
G_{31}&&&&\mu_{24}&&\\
24,20,12,8&&&&24&&\\
&\stackrel{\phantom{3}12}\searrow&&\stackrel{24\phantom{2}}{\nearrow}&
&&\\
&&G_{10}&&&&\\
&&24,12&&&&\\
\end{array}
\end{equation}
\caption{The sequence of reflection subquotients of $G_{31}$.
Note that $G_{31}$ is not a duality group while
the other three groups   are duality groups with $d_1=24$.
}\label{fig2}
\end{figure}

The article is organized as follows.
We first recall the definitions of regular elements, reflection subquotients,
admissible triplets and good basic invariants in \S \ref{sec:definitions}.
In  \S \ref{sec:reduction-theorem}, we prove the above results.
The rest of the article (\S 4 --\S 6)
is devoted to the examples depicted
in \eqref{monomial-sequence} and Figures \ref{fig1} and \ref{fig2}.

\subsection*{Note Added}
After this article was submitted, a referee kindly pointed out that
there is a work \cite{Slodowy} by Slodowy in which 
some sequences of reflection subquotients appear 
in relation to deformations of simple singularities. 
The reduction sequence depicted in Figure \ref{fig1}
is compatible with the list of reflection subquotients of $E_6$ in \cite[Fig.\,6]{Slodowy}.
As for $E_8$ and $E_7$, reflection subquotients  are listed in Fig.\,4 and Fig.\,5 in \cite{Slodowy} respectively.
It would be nice if we also have concrete descriptions of good basic invariants 
for the reduction sequences which include $E_8$ and $E_7$.
We hope to address this problem in a future publication. 

\subsection*{Acknowledgment}
The work of Y.K.  is supported in part by
KAKENHI Kiban-S (16H06337) and KAKEHI Kiban-S (21H04994).
S.M. is supported in part by
KAKENHI Kiban-C (23K03099). 
Y.K. and S.M. thank Ikuo Satake and Atsushi Takahashi for valuable discussions and comments.
The authors thank the referee for useful comments and for bringing the reference \cite{Slodowy} to their attention.  
\section{Regular elements, reflection subquotients, admissible triplets and good basic invariants}
\label{sec:definitions}
\subsection{Notations}
In this section,
 $G$ denotes a finite complex reflection group acting on a complex vector space $V$
of dimension $n$.
The degrees of $G$  are denoted $d_1,\ldots, d_n$
and the codegrees of $G$ are denoted $d_1^*,\ldots, d_n^*$. 
Unless otherwise stated,
we assume that the degrees are  in the descending order and 
the codegrees are in the ascending order:
$$
d_1\geq d_2\geq \ldots \geq d_n,\quad
d_1^*(=0)\leq d_2^*\leq \ldots,\leq d_n^*~.
$$
We also assume that any set of basic invariants $x_1,\cdots, x_n$ is
taken so that $\mathrm{deg}\, x_{\alpha}=d_{\alpha}$.

For a positive integer $d$, we set
\begin{equation}\nonumber
\mathfrak{a}(d):=\#\{{\alpha}\mid \text{$d$ divides $d_{\alpha}$}\}
~,\quad
\mathfrak{b}(d):=\#\{{\alpha}\mid \text{$d$ divides $d_{\alpha}^*$}\}.
\end{equation}
We also put 
$$
\mathfrak{i}(d)=\{\alpha\in \{1,2,\ldots,n\}\mid \text{$d$ divides $d_{\alpha}$}
\}
,\quad
\mathfrak{i}^c(d)=\{1,2,\ldots,n\}\setminus\mathfrak{i}(d).
$$

For $g\in GL(V)$ and $\zeta \in \mathbb{C}$, $V(g, \zeta)$ denotes
the $\zeta$-eigenspace of $g$ in $V$. 

\subsection{Regular elements}
\label{sec:regularelements}
In this subsection, we recall  the definitions of regular vector, regular element
and the existence theorem for regular elements \cite{Springer} 
\cite[Ch.11]{LehrerTaylor}.

\begin{definition}
\begin{enumerate}
\item[(i)] A vector $q\in V$ is {\it regular} if it lies on no reflection hyperplanes of $G$.
\item[(ii)] An element $g\in G$ is {\it regular} if $g$ has an eigenspace 
$V(g, \zeta)$ which contains a regular vector. In this case, 
we say that $g$ is a $\zeta$-regular element.
If $\zeta$ is a primitive $d$-th root of unity,
then we also say that $g$ is a $d$-regular element. 
\end{enumerate}
\end{definition}

\begin{remark}\label{remark-regularelements}
\begin{enumerate}\item
It is known that the necessary and sufficient condition for the existence of 
a $d$-regular element of $G$ is $\mathfrak{a}(d)=\mathfrak{b}(d)$.
See \cite[Theorem 11.28]{LehrerTaylor}. 
\item
Eigenvalues of a $\zeta$-regular element $g$ are
$\zeta^{1-d_{\alpha}}$ $(1\leq \alpha\leq n)$.
In particular, $\zeta^{d_{\alpha}}=1$ for some $\alpha$
\cite[Theorem 11.56]{LehrerTaylor}.
\end{enumerate}
\end{remark}

\subsection{Reflection subquotients in the regular cases}
We collect necessary facts on reflection subquotient.

For a subspace $E\subset V$, 
we put
\begin{equation}
\begin{split}
N_E&=\{s\in G\mid s(E)\subset E\},
\quad
C_E=\{s\in G\mid s(v)=v \text{ for all } v\in E\}.
\end{split}
\end{equation}
The map from $N_E$ to $GL(E)$ obtained by
restricting the action of each element $s\in N_E$ to $E$ is denoted $\pi_E$.
 The surjective homomorphism
 $S[V]\to S[E]$ of polynomial rings induced by the inclusion $E\to V$ is denoted $\phi_E$.

The following results can be found in \cite[Lemma 11.14, Corollary 11.17, Theorem 11.24,
Theorem 11.33, Theorem 11.38, Theorem 11.39]{LehrerTaylor}.

\begin{proposition}\label{LT}
Let $G$ be a finite complex reflection group.
Let $\delta$ be a regular number for $G$, i.e.
a positive integer satisfying $\mathfrak{a}(\delta)=\mathfrak{b}(\delta)$.
Let $h\in G$ be a $\xi$-regular element with respect to a $\delta$-th primitive root of unity $\xi$ and
let $E=V(h,\xi)$. 
Then the followings hold.
\begin{enumerate}
\item $\dim E=\mathfrak{a}(\delta)$.
\item 
$N_E=\{s\in G\mid sh=hs\}$.
\item $C_E=\{1\}$. Thus $\pi_E(N_E)$ acts faithfully on $E$. 

\item 
$\pi_E(N_E)$ is  a finite complex reflection group on $E$. 
Its reflection hyperplanes are the intersections with $E$
of the reflection hyperplanes of $G$ not containing $E$. 

\item 
If $x=(x_1,\ldots, x_n)$ is a set of basic invariants of $G$,
then $\phi_E (x_{\alpha})=0$ for $\alpha\in \mathfrak{i}^c(\delta)$.
Moreover $\phi_E(x_{\alpha})$ $({\alpha}\in \mathfrak{i}(\delta))$ form a set of basic invariants
of $\pi_E(N_E)$. In particular,  the degrees of $\pi_E(N_E)$
are $d_{\alpha}~~(\alpha\in \mathfrak{i}(\delta))$.

\item \label{irred}
If $G$ is an irreducible finite complex reflection group, so is $\pi_E(N_E)$.
\item
The codegrees of $\pi_E(N_E)$ are those  of $G$ which are divisible by $\delta$.
\end{enumerate}
\end{proposition}

\begin{remark}
It is known that the triviality of $C_E$  is equivalent to the regularity of $E$ (\cite[Lemma 11.22]{LehrerTaylor}).
Even if $E=V(h, \xi)$ is not regular, it is known that
if $E=V(h, \xi)$ is a maximal $\xi$-subspace 
then $\pi_E(N_E)\cong N_E/C_E$ is a complex reflection group
on $E$. See \cite[Ch 11]{LehrerTaylor} and the references therein.
By this reason, the reflection group $\pi_E(N_E)$ in Proposition \ref{LT} 
is called {\it a reflection subquotient}, although we only treat the cases 
of regular $E$ in the rest of the article.
\end{remark}

If both $\xi,\tilde{\xi}$ are $\delta$-th primitive roots of unity,
the reflection subquotients by 
a $\xi$-regular element $h$ and that by a $\tilde{\xi}$-regular element  $\tilde{h}$
are conjugate \cite[Proposition 11.18]{LehrerTaylor}. 
So the reflection subquotient $\pi_E(N_E)$ is uniquely determined 
up to conjugacy by the regular number $\delta$. 
It is denoted $G_{\delta}$ and called  {\it the $\delta$-reflection subquotient} in this article.

\subsection{Admissible triplet and good basic invariants}
In this subsection, we recall the definitions of admissible triplets and good basic invariants
 introduced by Satake \cite{Satake2020}.
\begin{definition}
A triple $(g,\zeta,q)$ consisting of 
\begin{itemize}
\item a primitive $d_1$-th root of unity $\zeta$
(where $d_1$ is the largest degree of $G$),
\item a $\zeta$-regular element $g\in G$,
\item a regular vector $q$ satisfying $gq=\zeta q$,
\end{itemize}
is called {\it an admissible triplet} of $G$.
\end{definition}
As stated in \S \ref{sec:regularelements},
the necessary and sufficient condition
for the existence of an admissible triplet is $\mathfrak{a}(d_1)=\mathfrak{b}(d_1)$.

\begin{definition}
Let $(g,\zeta,q) $ be an admissible triplet of $G$.
We say that  a linear coordinate system $z=(z_1,\ldots, z_n)$  of $V$ is a $(g,\zeta)$-graded coordinate system if 
$$
g^* z_{\alpha}=\zeta^{d_{\alpha}-1} z_{\alpha} \quad (1\leq \alpha\leq n)
$$
holds.
\end{definition}
Given a $\zeta$-regular element $g$,
the eigenspace decomposition gives a basis of $V$ consisting of 
eigenvectors of $g$. If one takes the associated coordinate system, 
it is $(g,\zeta)$-graded.

\begin{definition}
Let $(g,\zeta,q)$ be an admissible triplet of $G$
and let $z=(z_1,\ldots, z_n)$  be a $(g,\zeta)$-graded coordinate system.
We say that a set of basic invariants $x=(x_1,\ldots, x_n)$ is {\it compatible} 
at $q$ with $z$ if
\begin{equation}\label{compatible-x0}
\frac{\partial x_{\alpha}}{\partial z_{\beta}}(q)=\begin{cases}
1&(\alpha=\beta)
\\
0&(\alpha\neq \beta)\end{cases}
\quad (1\leq \alpha,\beta\leq n)~.
\end{equation}
\end{definition}
For any admissible triplet $(g,\zeta,q)$ and any  $(g,\zeta)$-graded coordinate system $z$,
there exists a set of basic invariants  which  is compatible with $z$  at $q$.

For $1\leq \alpha\leq n$, we set
\begin{equation}\label{defIalpha}
\begin{split}
I_{\alpha}^{(0)}&=\{
(a_1,\ldots,a_n)\in \mathbb{Z}_{\geq 0}^n\mid 
a_1d_1+\cdots+a_nd_n
=d_{\alpha},~~a_1+\cdots+a_n\geq 2
\}
~.
\end{split}\end{equation}

\begin{definition}\label{def-good} 
A set of basic invariants $x=(x_1,\ldots, x_n)$ of $G$
is {\it good} with respect to an admissible triplet $(g,\zeta,q)$
if
\begin{equation}\label{good0}
\frac{\partial^{a_1+\cdots+a_n} x_{\alpha}}{\partial z_1^{a_1}\cdots\partial z_n^{a_n}} (q)
=0 \quad (1\leq \alpha\leq n,~~ (a_1,\ldots,a_n)\in I_{\alpha}^{(0)}),
\end{equation}
where $z=(z_1,\ldots, z_n)$ is a $(g,\zeta)$-graded coordinate system of $V$.
\end{definition}
\begin{remark}\label{remark:uniqueness-gbi}
The definition of good basic invariants does not depend on the choice of
the graded coordinate system
in the sense that
any two sets of good basic invariants with respect to the same admissible triplet
are $\mathbb{C}$-linear combinations of each other \cite{Satake2020, KM2023}.
Moreover,  if $\mathfrak{a}(d_1)=\mathfrak{b}(d_1)=1$, 
the set of good basic invariants does not depend on the choice of admissible triplet
in the sense that
any two sets of good basic invariants with respect to (possibly different) admissible triplets
are $\mathbb{C}$-linear combinations of each other.
\end{remark}

\section{Reduction Theorem}\label{sec:reduction-theorem}

\subsection{Admissible triplet and reflection subquotients}
\label{sec:at-rq}
In this section,
$G$ is a finite complex reflection group satisfying $\mathfrak{a}(d_1)=\mathfrak{b}(d_1)$
and $(g,\zeta,q)$ is an admissible triplet of $G$.
Let $\delta>1$ be a divisor of $d_1$.
Then
$\zeta^{d_1/\delta}$ is a  primitive $\delta$-th root of unity,
and $g^{d_1/\delta}$ is $\zeta^{d_1/\delta}$-regular.
So we can apply Proposition \ref{LT} to this setting.
Let $E=V(g^{d_1/\delta},\zeta^{d_1/\delta})$
and let 
$N_E$, $\pi_E$, $\phi_E$ be as defined in Proposition \ref{LT}.

Then  next statements easily follow from Proposition \ref{LT}.
\begin{itemize}
\item $G_{\delta}:=\pi_E(N_E)$ is a reflection group on $E$.
\item The highest degree of $G_{\delta}$ is $d_1$. 
\item  $q\in E$ since $gq=\zeta q$ implies $g^{d_1/\delta}q=\zeta^{d_1/\delta}q$.
\item $g\in N_E$ since  $g$ commutes with $g^{d_1/\delta}$.
\item $(\pi_E(g),\zeta,q)$ is an admissible triplet  of $G_{\delta}$.
\end{itemize}

\subsection{Good basic invariants of $G_{\delta}$}
\begin{theorem}\label{thm:gbi}
\begin{enumerate}
\item If $z=(z^1,\ldots, z^n)$ is a $(g,\zeta)$-graded coordinate system of $V$,
$\phi_E (z_{\alpha})=0$ $(\alpha\in \mathfrak{i}^c(\delta))$ on $E$.
Moreover
$\phi_E(z_{\alpha})$ $(\alpha\in \mathfrak{i}(\delta))$ form a
$(\pi_E(g),\zeta)$-graded coordinate system of $E$.
In other words, the ring homomorphism $\phi_E:S[V]\to S[E]$ is given by
\begin{equation}\label{ring-hom}
S[V]\cong \mathbb{C}[z_1,\ldots, z_n]\longrightarrow
S[E]\cong \mathbb{C}[z_1,\ldots,z_n]/(z_{\alpha}\mid  \alpha\in \mathfrak{i}^c(\delta))~.
\end{equation}

\item If a set $x=(x_1,\ldots, x_n)$ of basic invariants of $G$ is 
compatible at $q$ with a $(g,\zeta)$-graded coordinate system $z$ of $V$,
then 
the set of basic invariants $\phi_E(x_{\alpha})$ $(\alpha\in \mathfrak{i}(\delta))$ of $G_{\delta}$
is compatible at $q$ with  
$\phi_E(z_{\alpha})$ $(\alpha\in \mathfrak{i}(\delta))$.

\item  If a set $x=(x^1,\ldots, x^n)$ of basic invariants of $G$ is 
good with respect to $(g,\zeta,q)$, 
then the set of basic invariants $\phi_E(z_{\alpha})$ $(\alpha\in \mathfrak{i}(\delta))$ of 
$G_{\delta}$
is good with respect to 
$(\pi_E(g),\zeta,q)$.

\end{enumerate}
\end{theorem}

\begin{proof}
(1)  
The basis  associated to  the  $(g,\zeta)$-graded coordinate system $z$ consists of  eigenvectors 
$q_1,\ldots,q_{n}$ of $g$ satisfying
$gq_{\alpha}=\zeta^{1-d_{\alpha}}q_{\alpha}$. 
Then $q_{\alpha}\in E$ if and only if  
$\zeta^{(1-d_{\alpha})d_1/\delta}=\zeta^{d_1/\delta}$, or
$\zeta^{d_1d_{\alpha}/\delta}=1$. This condition is equivalent to the condition that $ d_{\alpha}$ is divisible by $\delta$.
Therefore 
\begin{equation}\label{E-z}
E=\bigoplus_{\alpha\in \mathfrak{i}(\delta)} \mathbb{C}q_{\alpha}~.
\end{equation}
This implies that  $\pi_E(z_{\alpha})$
($\alpha\in \mathfrak{i}(\delta)$)
 make a coordinate system of $E$ and that the coordinate system is $(\pi_E(g),\zeta)$-graded.
Eq.\eqref{E-z} also implies that
$z_{\alpha}=0$ on $E$ for $\alpha\in \mathfrak{i}^c(\delta)$.
So the subspace $E$ is the zero set of $z_{\alpha}$ ($\alpha\in \mathfrak{i}^c(\delta)$)
and  \eqref{ring-hom} follows.

In the proofs of (2)(3), we identify $\phi_E(z_{\alpha})$ ($\alpha\in \mathfrak{i}(\delta)$) with $z_{\alpha}$ to simplify the notation.
 \\
 (2) First recall from Theorem \ref{LT} that  $\phi_E(x_{\alpha})$ $(\alpha\in \mathfrak{i}(\delta))$ is a
 set of basic invariants of $G_{\delta}$.

 Notice that for $f\in S[V]$ and $\alpha\in \mathfrak{i}(\delta)$,
 \begin{equation}\label{phi-q-commutes}
 f(q)=\phi_E(f) (q),\quad \phi_E\left(\frac{\partial f}{\partial z_{\alpha}}\right)=\frac{\partial \phi_E(f)}{\partial z_{\alpha}}
 \end{equation}
 hold.
Therefore 
$$
\frac{\partial \, \phi_E (x_{\beta})}{\partial z_{\alpha}}(q) =
\frac{\partial x_{\beta}}{\partial z_{\alpha}}(q)=
\begin{cases}
1&(\alpha=\beta)\\
0&(\alpha\neq \beta)
\end{cases}
\qquad (\alpha,\beta\in\mathfrak{i}(\delta))
$$
holds by the compatibility of $x$ with $z$.
\\
(3)  Let us set 
\begin{equation}\nonumber
I_{\alpha}^{(0)}(\delta)=\{(a_1,\ldots a_n)\in I_{\alpha}^{(0)}\mid a_{\mu}=0 \text{ for all }\mu \in \mathfrak{i}^c(\delta)\} \quad (1\leq \alpha\leq n)~.
\end{equation}
For $\beta\in \mathfrak{i}(\delta)$ and  $(a_1,\ldots, a_n)\in I_{\alpha}^{(0)}(\delta)$,
$$
\frac{\partial^{a_1+\cdots+a_n} \, \phi_E (x_{\beta})}{\partial z_1^{a_1}
\cdots \partial z_n^{a_n}}(q) 
=\frac{\partial^{a_1+\cdots+a_n} \, x^{\beta}}{\partial z_1^{a_1}\cdots \partial z_n^{a_n}}(q)
$$
holds because  of \eqref{phi-q-commutes}. Given that $x$ is good, it is immediate to see that the right-hand-side vanishes.
\end{proof}

\subsection{Potential vector field of $G_{\delta}$}
A duality group is 
an irreducible finite complex reflection group  satisfying the relation 
\cite[\S B4]{OrlikTerao} \cite[\S 12.6]{LehrerTaylor}
\begin{equation}\label{duality-degree}
d_{\alpha}+d_{\alpha}^*=d_1\quad (1\leq \alpha\leq n).
\end{equation}
It is known that $\mathfrak{a}(d_1)=\mathfrak{b}(d_1)=1$ holds
for all duality groups.

 In this subsection, first we recall the potential vector field of the natural Saito structure 
 for duality groups \cite[\S 7.5]{KMS2018}.
 \begin{definition}\label{def:pvf}
 Let $(g,\zeta,q)$ and $z=(z_1,\ldots, z_n)$ be an admissible triplet of 
 a duality group $G$ acting on $V$ and  a $(g,\zeta)$-graded coordinate  system of  $V$.
Let $x=(x_1,\ldots, x_n)$ be a  set  of good basic invariants with  respect to 
$(g,\zeta,q)$. It is also assumed that $x$ is compatible with $z$ at $q$.
The potential vector field of $G$ is an $n$-tuple $(\mathcal{G}_1,\ldots,\mathcal{G}_n)$ 
of the $G$-invariant polynomials
\begin{equation}\label{eq:pvf}
\mathcal{G}_{\gamma}=
\frac{z_1(q)}{d_{\gamma}-1}
\sum_{(a_1,\ldots, a_n)\in I_{\gamma}^{(1)}} 
\frac{\partial^{a_1+\cdots+a_n} x_{\gamma}}
{\partial z_1^{a_1}\cdots \partial z_n^{a_n}} (q) \cdot \frac{x_1^{a_1}\cdots x_n^{a_n}}{a_1!\cdots a_n!}
\quad (1\leq \gamma\leq n)~.
\end{equation}
Here $$
I^{(1)}_{\gamma}=\{(a_1,\ldots, a_n)\in \mathbb{Z}^n_{\geq 0}\mid 
a_1d_1+\cdots+a_nd_n=d_{\gamma}+d_1, ~~a_1+\cdots+a_n\geq 2\}~.
$$
\end{definition}

Let $G$ be a duality group and $\delta>1$ be a divisor of $d_1$.
Then notice that the reflection subquotient $G_{\delta}$ is also a duality group.
\begin{theorem}\label{thm:pvf}
Using the notations of \S \ref{sec:at-rq} and Definition \ref{def:pvf},
we have  the followings.
\begin{enumerate}
\item $\phi_E(\mathcal{G}^{\gamma})=0$ for $\gamma\in \mathfrak{i}^c(\delta)$.
\item $\phi_E(\mathcal{G}^{\gamma})$  $(\alpha\in \mathfrak{i}(\delta))$ 
form a potential vector field of $G_{\delta}$.
\end{enumerate}
\end{theorem}
\begin{proof}
(1) Let us set 
\begin{equation}\nonumber
I_{\gamma}^{(1)}(\delta)=\{(a_1,\cdots,a_n)\in I_{\gamma}^{(1)}\mid a_{\mu}=0 \text{ for all }\mu \in \mathfrak{i}^c(\delta)\} \quad (1\leq \gamma\leq n)~.
\end{equation}
Because of the vanishing $\phi_{E}(x_{\alpha})=0$ ($\alpha\in \mathfrak{i}^c(\delta)$),
we have
\begin{equation}\begin{split}\nonumber
\phi_E(\mathcal{G}^{\gamma})&=
\frac{z^1(q)}{d_{\gamma}-1}
\sum_{(a_1,\ldots,a_n)\in I_{\gamma}^{(1)}(\delta) } 
\frac{\partial^{a_1+\cdots+a_n} x_{\gamma}}
{\partial z_1^{a_1}\cdots \partial z_n^{a_n}} (q) 
\cdot \frac{\phi_E(x_1^{a_1}\cdots x_n^{a_n})}{a_1!\cdots a_n!}
\end{split}
\end{equation}
Moreover, in the right-hand-side, 
$$
\frac{\partial^{a_1+\cdots+a_n} x^{\gamma}}
{\partial z_1^{a_1}\cdots \partial z_n^{a_n}} (q) 
\stackrel{\eqref{phi-q-commutes}}{=}
\frac{\partial^{a_1+\cdots+a_n} \phi_E(x^{\gamma})}
{\partial z_1^{a_1}\cdots \partial z_n^{a_n}} (q).
$$
Thus we obtain the expression
\begin{equation}\begin{split}\label{phi-g}
\phi_E(\mathcal{G}^{\gamma})&=
\frac{z^1(q)}{d_{\gamma}-1}
\sum_{(a_1,\ldots,a_n)\in I_{\gamma}^{(1)}(\delta)} 
\frac{\partial^{a_1+\cdots+a_n} \phi_E(x_{\gamma})}
{\partial z_1^{a_1}\cdots \partial z_n^{a_n}} (q) 
\cdot \frac{\phi_E(x_1^{a_1}\cdots x_n^{a_n})}{a_1!\cdots a_n!}~.
\end{split}
\end{equation}
For $\gamma\in \mathfrak{i}^c(\delta)$, we have $\phi_E(x^{\gamma})=0$ and hence
$ 
\phi_E(\mathcal{G}^{\gamma})=0$.
For $\gamma\in \mathfrak{i}(\delta)$, \eqref{phi-g} is nothing but the potential vector field 
of $G_{\delta}$.
\end{proof}

\begin{remark} \label{remark-pvf}
Under the change of basic invariants,  the vector field
$$
x_1^{-1}\sum_{\gamma=1}^n \mathcal{G}_{\gamma}\frac{\partial}{\partial x_{\gamma}}
$$
is invariant. 
The factor $x^{-1}_1$ comes from the factor $\Omega_1^{-1}$ in the definition of 
$C_{\alpha\beta}^{\gamma}$
of \cite[Lemma 7.4]{KMS2018}. 
\end{remark}

\section{Reflection subquotients of  $G(m, m, n+1)$ and $G(m,1,n)$}
\label{sec:monomial-groups}
In this section,
we first recall 
admissible triplets 
for $G(m,m,m+1)$ and $G(m,1,n)$ used in \cite[\S 10]{KM2023}.
Then
we consider reflection subquotients of the duality groups $G(m,m,n+1)$ and 
$G(m,1,n)$. 
For the definitions of the groups $G(m,m,n+1)$, $G(m,1,n)$ and their degrees,
basic invariants and so on, see \cite[\S 2]{LehrerTaylor}.

The cyclic group consisting of $m$-th roots of unity is denoted $\mu_m$.
The standard basis of  $\mathbb{C}^n$ is denoted $\{e_1,\ldots, e_n\}$ and 
the associated linear coordinates are denoted $u_1,\ldots, u_n $. 
In this section, $\zeta$ is a primitive $nm$-th root of unity.
\subsection{$G(m,m,n+1)$ for $m\geq 2$, $ n\geq 1$}
\label{sec:mmn}
Let $$\mathcal{A}(m,m,n+1)
=\{(\theta_1,\ldots,\theta_{n+1},\sigma) \in \mu_m^{n+1}\times \mathfrak{S}_{n+1}\mid 
\theta_1\cdots \theta_{n+1}=1\}.
$$ 
As a set, the group $G(m,m,n+1)$ is $\mathcal{A}(m,m,n+1)$.
Its group structure as well as the action on $\mathbb{C}^{n+1}$ is given by the following  injection $\iota_{n+1}:\mathcal{A}(m,m,n+1)\to GL_{n+1}(\mathbb{C})$:
$$
\iota_{n+1}(\theta_1,\ldots,\theta_{n+1},\sigma)e_i
=\theta_{\sigma(i)}e_{\sigma(i)}.
$$
For example, if $n=2$ and if $\sigma$ is the cyclic permutation $(3~2~1)$, 
$$
\iota_{n+1}(\theta_1,\theta_2,\theta_3,\sigma)=
\begin{bmatrix}
0&\theta_1&0\\
0&0&\theta_2\\
\theta_3&0&0
\end{bmatrix}~.
$$
$G(m,m,n+1)$ is a duality group of rank $n+1$
and the degrees  are 
$m,2m,\ldots nm$ and $n+1$.
For convenience, we arrange the degrees in the following manner:
\begin{equation}\label{Gmmn-degree}
d_1=nm, ~~d_2=(n-1)m,~~\ldots, d_{n-1}=2m,~~d_n=m,\quad d_{n+1}=n+1~.
\end{equation}

Let
$$
g:=\iota_{n+1}(\underbrace{1,\ldots,1}_{n-1},\zeta^n,\zeta^{-n},(n,\cdots,2,1))
=
\left[
\begin{array}{c|c}
\begin{array}{cc}
0&I_{n-1}\\
\zeta^{n}&0
\end{array}
&0\\
\hline
0&\zeta^{-n}
\end{array}
\right].
$$
Its eigenvalues are
$\lambda_i=\zeta^{1+(i-1)m}$ ($i=1,\ldots,n$)
and $\zeta^{-n}$,  and eigenvectors are
$$
q_i=\frac{1}{\sqrt{n}} \begin{bmatrix}
1\\ \lambda_i\\ \lambda_i^2\\ \vdots \\ \lambda_i^{n-1}\\0
\end{bmatrix}\quad
(1\leq i \leq n),\qquad
q_{n+1}=\begin{bmatrix}0\\ \vdots\\0 \\1\end{bmatrix}.
$$
Since reflection hyperplanes of $G(m,m,n+1)$ are  the orthogonal complements of 
$$
e_i-\zeta^{kn} e_j \quad (1\leq i<j\leq n+1,~ 0\leq k<m)~,
$$
it is not difficult to see that the vector $q_1$ does not lie on any reflection hyperplanes.
Therefore  $(g, \zeta, q_1)$ is an admissible triplet of $G(m,m,n+1)$.

It is known that 
a set of basic invariants of $G(m,m,n+1)$ is given by
\begin{equation}\nonumber \label{sigma}
\sigma_{\alpha}:=\begin{cases}
\mathbb{E}_{n+1-\alpha}(u_1^m,\ldots, u_{n+1}^m)
\quad (1\leq \alpha \leq n),\\
\mathbb{E}_{n+1}(u_1,\ldots, u_{n+1})=u_1\cdots u_{n+1}\quad (\alpha=n+1).
\end{cases}
\end{equation}
Here $\mathbb{E}_k(v_1,\ldots, v_{n+1})$ denotes the $k$-th elementary symmetric polynomial
in the $(n+1)$ variables $v_1,\ldots, v_{n+1}$.
It seems it is not feasible to obtain closed expressions of the good basic invariants and the potential vector field for general $n$. For $G(m,m,2)$, $G(m,m,3)$ and $G(m,m,4)$, see \cite[\S 10.1]{KM2023}.

\subsection{$G(m,1,n)$ for $m\geq 2$}
\label{sec:m1n}
As a set, the group $G(m,1,n)$ is $\mu_m^{n}\times \mathfrak{S}_{n}$.
Its group structure and the action on $\mathbb{C}^{n}$ is given by the following  injection $\iota_{n}:
G(m,1,n)\to GL_{n}(\mathbb{C})$:
$$
\iota_{n}(\theta_1,\ldots,\theta_{n},\sigma)e_i
=\theta_{\sigma(i)}e_{\sigma(i)}.
$$
$G(m,1,n)$ is a duality group of rank $n$ and the degrees are $m,2m,\ldots, nm$.
As in the case of $G(m,m,n+1)$,
we arrange the degrees in the following manner:
$$
d_1=nm, ~~d_2=(n-1)m,~~\ldots, d_{n-1}=2m,~~d_n=m.
$$

Let 
$$
\bar{g}:=
\iota_n(\underbrace{1,\ldots,1}_{n-1},\zeta^n,(n~\ldots~2~1))=
\left[
\begin{array}{c|c}
0&I_{n-1}\\ \hline
\zeta^n&0
\end{array}
\right].
$$
 Its eigenvalues are
$\lambda_i=\zeta^{1+(i-1)m}$ ($i=1,\ldots,n$) 
and $\lambda_i$-eigenvectors are
$$
\bar{q}_i=\frac{1}{\sqrt{n}} \begin{bmatrix}
1\\ \lambda_i\\ \lambda_i^2\\ \vdots \\ \lambda_i^{n-1}
\end{bmatrix}.
$$
Since reflection hyperplanes of $G(m,1,n)$ are orthogonal complements of 
$$
e_i-\zeta^{kn}e_j ~~ (1\leq i<j\leq n,0\leq k<m),\quad
e_i ~~(1\leq i\leq n),
$$
$\bar{q}_1$ does not lie on any of these. Therefore 
$\bar{q}_1$ is a regular vector   and $(\bar{g}, \zeta, \bar{q}_1)$
is an admissible triplet of $G(m,1,n)$.

Let 
$$
\bar{\sigma}_{\alpha}:=\mathbb{E}_{n+1-\alpha}(u_1^m,\ldots, u_n^m)
\quad  (\alpha=1,\ldots,n). 
$$
Then $\bar{\sigma}_1,\ldots,\bar{\sigma}_n$
form a set of basic invariants of $G(m,1,n)$. 
For the good basic invariants and the potential vector field of $G(m,1,2)$, $G(m,1,3)$, $G(m,1,4)$, see \cite[\S 10.2]{KM2023}.

\subsection{The reflection subquotient of $G(m,m,n+1)$ by $\delta=m$}

\begin{proposition}\label{prop:mmn2m1n}
If $n+1$ is not divisible by $m$, 
the $m$-reflection subquotient of $G(m,m,n+1)$ is $G(m,1,n)$.
\end{proposition}

\begin{proof} 
Take the admissible triplet $(g,\zeta,q_1)$ constructed in \S \ref{sec:mmn}.
Notice that  $d_1/m=n$.
Since 
$$
g^n=\begin{bmatrix}\zeta^n I_{n}&0\\0&\zeta^{-n^2}\end{bmatrix},
$$
the $\zeta^n$-eigenspace $E$ of $g^n$  is 
$$
E=\sum_{j=1}^n \mathbb{C}e_j~. 
$$
(Notice that we need the assumption here; if  $n+1$ is divisible by $m$,  $E=\mathbb{C}^{n+1}$.)
So the subgroup $N_E=\{s\in G(m,m,n+1)\mid s(E)=E\}$ consists of matrices
$
\iota_{n+1}(\theta_1,\ldots,\theta_n,\theta_{n+1},\sigma)$ such that  $\sigma(n+1)=n+1$.
Then we immediately see that the $m$-reflection subquotient of
$
G(m,m,n+1)$ is $G(m,1,n).
$
\end{proof}

\begin{remark}
In the case where $n+1$ is divisible by $m$,  
the $\zeta^n$-eigenspace of $g^n$ is $\mathbb{C}^{n+1}$
and 
the $m$-reflection subquotient of $G(m,m,n+1)$ is itself.
However,  if we take $E=\sum_{j=1}^n \mathbb{C}e_j$ instead of the $\zeta^n$-eigenspace of $g^n$, $N_E=G(m,1,n)$ hold. Moreover the statements of Proposition \ref{LT}(5) and 
 Theorems \ref{thm:gbi},\ref{thm:pvf} hold true if we replace $\mathfrak{i}(\delta)$ with $\{1,2,\ldots,n\}$
and $\mathfrak{i}^c(\delta)$ with $\{n+1\}$.
\end{remark}

Proposition \ref{prop:mmn2m1n} and Theorems \ref{thm:gbi} and \ref{thm:pvf} imply that 
one can obtain 
the good basic invariants and the potential vector field of $G(m,1,n)$ from those of $G(m,m,n+1)$.
As an example, the reader can compare those of $G(m,m,3)$ and $G(m,1,2)$ listed in \cite[\S 10.1, \S 10,2]{KM2023}. 
Moreover,
the following lift  of good basic invariants from $G(m,1,n)$ to $G(m,m,n+1)$ holds.
\begin{proposition}\label{lift}
Let  $\bar{x}=(\bar{x}_1,\ldots, \bar{x}_n)$ be the set of good basic invariants
of 
$G(m,1,n)$ with respect to $(\bar{g},\zeta, \bar{q}_1)$ which is compatible with 
the graded coordinates $\bar{z}_1,\ldots,\bar{z}_n$.  If $\bar{x}_{\alpha}$ $(1\leq \alpha\leq n)$ is expressed as
$$\bar{x}_{\alpha}=\sum_{
\begin{subarray}{c}a=(a_1,\ldots,a_n)\in \mathbb{Z}_{\geq 0}^n;\\
a_1d_1+\cdots+a_nd_n=d_{\alpha}\end{subarray}}
c_a^{\alpha} \cdot \bar{\sigma}_1^{a_1}\cdots\bar{\sigma}_n^{a_n}
\quad (c_a^{\alpha}\in \mathbb{C}),
$$
then 
\begin{equation}\label{eq:invariants}
\begin{split}
x_{\alpha}&=\sum_{\begin{subarray}{c}a=(a_1,\ldots, a_n)\in \mathbb{Z}_{\geq 0}^n;\\
a_1d_1+\cdots+a_nd_n=d_{\alpha}\end{subarray}}
c_a^{\alpha}\cdot \sigma_1^{a_1}\cdots \sigma_n^{a_n}, 
\\
x_{n+1}&=(\sqrt{n})^n\zeta_{nm}^{-\frac{n(n+1)}{2}}\sigma_{n+1}
\end{split}
\end{equation}
form a set of good basic invariants of $G(m,m,n+1)$ with respect to $(g,\zeta,q_1)$
which is compatible with the graded coordinates $\bar{z}_1,\ldots,\bar{z}_n,u_{n+1}$.
\end{proposition}
The proof is elementary and omitted.

\subsection{Reflection subquotients of  $G(m,1,n)$ }
Since all the degrees of $G(m,1,n)$ are divisible by $m$,
considering the reflection subquotients  by  divisors of $d_1=nm$ 
is the same as considering those  by $m$ times divisors of $n$.
So let $k>1$ be a divisor of $n$. 
\begin{proposition}
The $km$-reflection subquotient of 
 $G(m,1,n)$ is isomorphic to $G(km,1,\frac{n}{k})$.
\end{proposition}

\begin{proof}
Take the admissible triplet $(\bar{g},\zeta,\bar{q}_1)$ constructed in \S  \ref{sec:m1n}.
Notice that $d_1/(km)=n/k=:n'$.

First we study the $\zeta^{n'}$-eigenspace $E$ of $\bar{g}^{n'}$.
The eigenvectors 
$\bar{q}_{\alpha}$ ($1\leq \alpha\leq n$) of $\bar{g}$ given in \S \ref{sec:m1n}
has eigenvalue $\zeta^{1+(\alpha-1) m}$.
Therefore $\bar{q}_{\alpha}\in E$ is equivalent to 
$(\alpha-1)n'm\equiv 0 \mod nm$, or
$\alpha-1\equiv 0 \mod k$. So we have
$$
E=\bigoplus_{\beta=0}^{n'-1} \mathbb{C} \bar{q}_{1+k\beta}~.
$$
Notice that $\bar{q}_{1+k\beta}$ is written as follows.
\begin{equation}\nonumber\begin{split}
\bar{q}_{1+k\beta}&=\frac{1}{\sqrt{n}}\sum_{i=0}^{n-1} (\zeta^{1+km\beta })^i \, e_{i+1}
\stackrel{i=j+pn'}{=}\frac{1}{\sqrt{n}}
\sum_{j=0}^{n'-1}\zeta^{(1+km\beta)j}\sum_{p=0}^{k-1}\zeta^{pn'}\, e_{j+pn'+1}~.
\end{split}
\end{equation}
Therefore if we put
$$
\check{e}_j=\frac{1}{\sqrt{k}}\sum_{p=0}^{k-1}\zeta^{pn'}\,e_{j+pn'} \quad (1\leq j\leq  n'),
$$
we have
$$
\bar{q}_{1+k\beta}=\frac{\sqrt{k}}{\sqrt{n}}\sum_{j=1}^{n'}\zeta^{(1+km\beta)(j-1)} \,\check{e}_j~.
$$
Thus  we have
$$
E=\bigoplus_{j=1}^{n'} \mathbb{C}\check{e}_j~.
$$

Next we study the subgroup $N_E=\{s\in G(m,1,n)\mid s(E)=E\}$. 
Let $\kappa:GL_{n'}(\mathbb{C})\to GL_n(\mathbb{C})$ be the following injection;
$$
\kappa(X)=\begin{bmatrix}
X&&O\\
  & \ddots&\\
 O &&X
\end{bmatrix}~.
$$
If $X=\iota_{n'}(\theta_1,\ldots,\theta_{n'},\sigma)$, then
$$
\kappa(X) e_{j+pn'}=\theta_{\sigma(j)} e_{\sigma(j)+pn'} \quad (1\leq j\leq n', 0\leq p<k )
$$
and 
$$
\kappa(X) \check{e}_j =\frac{1}{\sqrt{k}}\sum_{p=0}^{k-1}\zeta^{pn'} \theta_{\sigma(j)}\,e_{\sigma(j)+pn'}=\theta_{\sigma(j)}\check{e}_{\sigma(j)}~.
$$
Therefore $\kappa(X)$ preserves the subspace $E$. In other words,
$\kappa(G(m,1,n))\subset N_E$.

 On the other hand, by Proposition \ref{LT}-(4), 
 the degrees of the $km$-reflection subquotient $G(m,1,n)_{km}$ are $km,2km,\ldots, n'km$.
 Since the order of the reflection group equals the product of its degrees,
 $$
 |G(m,1,n)_{km}|=(km)\cdot (2km)\cdots (n'km)=|G(km,1,n')|~.
 $$
 Thus $G(m,1,n)_{km}\cong G(km,1,n')$.
 \end{proof}

As an example, 
the reader can check that the set of good basic invariants of $G(m,1,4)$  \cite[\S 10.2.3]{KM2023}
induce  that of $G(2m,1,2)$ \cite[\S 10.2.1]{KM2023}:
If we set 
$$\bar{\sigma}_1\to \bar{\sigma}_1,~~
\bar{\sigma}_2\to 0,~~
\bar{\sigma_3}\to \bar{\sigma}_2,~~\bar{\sigma}_4\to 0,$$
in  $x_1,x_3$ of $G(m,1,4)$,  $x_1,x_3$ are mapped to $x_1,x_2$ of $G(2m,1,2)$.

\subsection{The reduction sequence of $G(m,m,n+1)$}
Let $\delta$ be a divisor of $d_1=nm$.
\begin{itemize} 
\item
 If $\delta$ is a divisor of both $m$ and $n+1$,  the $\delta$-reflection subquotient is
$G(m,m,n+1)$ itself.
\item
 If $\delta$ is a divisor of $m$ and if $n+1$ is not divisible by $\delta$,
then by construction,
the $\delta$-reflection subquotient is the same as the $m$-reflection subquotient $G(m,1,n)$.
\item Since $\delta>1$ is assumed to be a divisor of $d_1=nm$,
it cannot happen that $\delta$ is not a divisor of $m$ but that $\delta$ is a divisor of $n+1$.

\item
 If both $m$ and $n+1$ are not divisible by $\delta$, let us write 
$\delta=k\cdot m'$ where $m'$ is the greatest common divisor of $\delta$ and $m$.
Notice that $k>1$ must be a divisor of $n$, since $\delta$ is a divisor of $nm$.
For such $\delta$,  the $km'$-reflection subquotient 
is the compotision of the $m'$-reflection subquotient which is $G(m,1,n)$
and the $k$-reflection subquotient.
This process is described in \eqref{monomial-sequence}.
\end{itemize}

\section{Sequence of reflection subquoitents of $G_{35}=E_6$}
\label{sec:E6}
In this section and the next, 
we study the sequences of reflection subquotients depicted in Figures \ref{fig1},\ref{fig2}
respectively.
To avoid double subscripts  such as $(G_{\bullet})_{\delta}$,
we write $G_{\bullet}(\delta)$ for the $\delta$-reflection subquotient of a primitive group $G_{\bullet}$.
For each group, we construct
an admissible triplet, a set of good basic invariants and the potential vector field.
For each arrow  $$
G_{\bullet}\stackrel{\delta}{\longrightarrow} G_{\bullet\bullet},
$$
we explicitly construct an isomorphism between the reflection subquoient $G_{\bullet}({\delta})$
and $G_{\bullet\bullet}$ (except the right-most  arrows $\longrightarrow \mu_{d_1}$).

The standard coordinates of $\mathbb{C}^n$ are denoted $u_1,\ldots, u_n$
and the standard basis is denoted $e_1,\ldots,e_n$. 
The reflection on $\mathbb{C}^n$ with a  root  $v\in \mathbb{C}^n$ and a nontrivial eigenvalue $\lambda\in \mathbb{C}^{\times} (\lambda\neq 1)$
is denoted $s(v,\lambda)$. Explicitly,
$$
s(v,\lambda)(w)=w-(1-\lambda)\frac{(w,v)}{(v,v)} v~.
$$
Here $(~,~)$ is the standard Hermitian inner product on $\mathbb{C}^n$ given by
$$
(v,w)=\sum_{i=1}^n v_i\bar{w}_i~.
$$
When $\lambda=-1$, we omit $\lambda$ and write $s(v)$ for $s(v,-1)$. 

In this section, $$\zeta=e^{\frac{\pi i}{6}}, \quad \omega=e^{\frac{2\pi i}{3}}~.$$

\subsection{An admissible triplet, good basic invariants 
and the potential vector filed of $G_{35}=E_6$}
The finite Coxeter group $G_{35}=E_6$ has rank $6$ and degrees $12,9,8,6,5,2$.
It has $36$ reflections of order two.
The following six roots form a simple root system \cite{Mehta1988} (see also \cite[\S 3]{Talamini2018}):
\begin{equation}\nonumber
\begin{split}
\alpha_1^{(35)}&=e_2-e_3+\frac{1}{\sqrt{2}}e_5+\sqrt{\frac{3}{2}}e_6,
\quad
\alpha_2^{(35)}=e_3-e_4-\sqrt{2}e_5,
\\
\alpha_3^{(35)}&=2e_4,\qquad\qquad\qquad\qquad\qquad
\alpha_4^{(35)}=e_3-e_4+\sqrt{2}e_5,
\\
\alpha_5^{(35)}&=e_2-e_3-\frac{1}{\sqrt{2}}e_5-\sqrt{\frac{3}{2}}e_6,
\quad
\alpha_6^{(35)}=e_1-e_2-e_3-e_4.
\end{split}
\end{equation}
Then  a Coxeter element is given by
\begin{equation}\nonumber
g_{(35)}=s(\alpha_1^{(35)})s(\alpha_2^{(35)})s(\alpha_3^{(35)})
s(\alpha_4^{(35)})s(\alpha_5^{(35)})s(\alpha_6^{(35)})~.
\end{equation}
The Coxeter element
$g_{(35)}$ is $12$-regular with the following eigenvalues and eigenvectors:
\begin{equation}\nonumber
\begin{split}
\zeta:~~
& q_1^{(35)}=\frac{1}{4\sqrt{3-\sqrt{3}}}
\begin{bmatrix}-(1+i) (\sqrt{3}+i)\\ (1-i) (\sqrt{3}-1) \\ 
(1+i)  (\sqrt{3}-2-i)\\ 
(1-i) (\sqrt{3}-1)\\
-\sqrt{2}\\
\sqrt{6}\end{bmatrix}
\quad \zeta^4:~~
q_2^{(35)}=\frac{1}{2\sqrt{2}}
\begin{bmatrix}
0\\-i\sqrt{2}\\0\\ i \sqrt{2}\\ \sqrt{3}\\1
\end{bmatrix},
\end{split}
\end{equation}
\begin{equation}\nonumber
\begin{split}
\zeta^5:~~&
 q_3^{(35)}=\frac{1}{4\sqrt{3+\sqrt{3}}}
\begin{bmatrix}
(1+i) (\sqrt{3}-i )\\
-(1-i) (1+\sqrt{3})\\
-(1+i)  (\sqrt{3}+2+i)\\
-(1-i) (1+\sqrt{3})\\
-\sqrt{2}\\
\sqrt{6}\end{bmatrix},
\quad
\zeta^7:~~
 q_4^{(35)}=\frac{1}{4\sqrt{3+\sqrt{3}}}
\begin{bmatrix}
(1-i) ( \sqrt{3}+i)\\
-(1+i) (1+\sqrt{3})\\
-(1-i) (2 +\sqrt{3}-i)\\
  - (1+i)(1+\sqrt{3})\\
   -\sqrt{2}\\
   \sqrt{6}\end{bmatrix},
\end{split}
\end{equation}
\begin{equation}\nonumber
\begin{split}
\zeta^8:~~&
q_5^{(35)}=\frac{1}{2\sqrt{2}}
\begin{bmatrix}
0\\i\sqrt{2}\\0\\- i \sqrt{2}\\ \sqrt{3}\\1
\end{bmatrix},
\quad
\zeta^{11}:~~
 q_6^{(35)}=\frac{1}{4\sqrt{3-\sqrt{3}}}
\begin{bmatrix}
-(1-i) ( \sqrt{3}-i)\\
(1+i)(\sqrt{3}-1)\\
-(1-i)(2 -\sqrt{3}-i)\\
   (1+i)(\sqrt{3}-1)\\
   -\sqrt{2}\\
   \sqrt{6}\end{bmatrix}~.
\end{split}
\end{equation}
We take an admissible triplet $(g_{(35)},\zeta, q_1^{(35)})$
and the $(g_{(35)},\zeta)$-graded coordinates of $\mathbb{C}^6$ associated with  $q^{(35)}_1,\ldots, q^{(35)}_6$.
The set of basic invariants which is good with repect to $(g_{(35)},\zeta,q_1^{(35)})$
and which is compatible with the above graded coordinates at $q_1^{(35)}$ is 
\begin{equation}\nonumber
\begin{split}
x_1^{(35)}&=\frac{1}{20\sqrt{3}}\left(
P_{12}-\frac{209}{2304}P_8P_2^2-\frac{77}{576}P_6^2+\frac{2959}{165888}P_6 P_2^3
-\frac{121}{1440}P_5^2P_2
-\frac{737}{10616832}P_2^6
\right),
\end{split}
\end{equation}
$$
x_2^{(35)}=\frac{1}{14\sqrt{3}}\left(P_9-\frac{7}{120}P_5P_2^2\right),
\quad
x_3^{(35)}=\frac{3}{40\sqrt{2}}\left(P_8-\frac{7}{24}P_6P_2+\frac{385}{248832}P_2^4
\right),
$$
$$
x_4^{(35)}=\frac{1}{8\sqrt{6}}\left(
P_6-\frac{5}{576}P_2^3\right)
,\quad
x_5^{(35)}=-\frac{1}{20}P_5,\quad 
x_6^{(35)}=\frac{1}{24} P_2~.
$$
Here 
\begin{equation}\nonumber
P_m=\sum_{i=1}^{27} l_i^m\quad (m=12,9,8,6,5,2),
\end{equation}
where $l_1,\ldots, l_{27}$ are following $27$ linear polynomials \cite[\S 3]{Talamini2018}
(see also \cite{Mehta1988}):
\begin{equation}\nonumber
\begin{split}
&2\sqrt{\frac{2}{3}}u_6,\quad
\sqrt{\frac{2}{3}}(\pm \sqrt{3}{u}_5-u_6),
\\
&\pm u_2\pm u_4+\frac{\sqrt{3}u_5+u_6}{\sqrt{6}},\quad
\pm u_1\pm u_3+\frac{\sqrt{3}u_5+u_6}{\sqrt{6}},
\\
&\pm u_2\pm u_3-\frac{\sqrt{3}u_5-u_6}{\sqrt{6}},\quad
\pm u_1\pm u_4-\frac{\sqrt{3}u_5-u_6}{\sqrt{6}},
\\
&\pm u_3\pm u_4-\sqrt{\frac{2}{3}}u_6,\quad
\pm u_1\pm u_2-\sqrt{\frac{2}{3}}u_6.
\end{split}
\end{equation}

The potential vector field is  given as follows. For the  simplicity of expression, we write $x_{\alpha}$ instead of $x_{\alpha}^{(35)}$ here.
\begin{equation}\nonumber
\mathcal{G}_{\gamma}^{(35)}=\frac{\partial }{\partial x_{7-\gamma}}\mathcal{F}_{E_6}
 \quad (1\leq \gamma\leq 6),
\end{equation}
where
\begin{equation}\begin{split}\label{pot35}
\mathcal{F}_{E_6}&=\frac{1}{2} x_1^2 x_6
+x_1x_3 x_4 +x_1 x_2 x_5
\\&+\frac{x_2^2 x_3}{2 \sqrt{2}}
+\frac{1}{2} \sqrt{\frac{3}{2}} x_2^2 x_4 x_6
+\frac{1}{8} x_2^2 x_6^4
+\sqrt{\frac{3}{2}} x_2 x_3 x_5 x_6^2
\\&\quad 
+\frac{\sqrt{3}}{2}  x_2 x_4^2x_5
+\frac{x_2 x_4 x_5 x_6^3}{\sqrt{2}}
+\frac{x_2 x_5^3 x_6}{\sqrt{3}}
\\&+\frac{x_3^3 x_6}{3\sqrt{2}}
 +\frac{1}{2}x_3^2 x_5^2
+\frac{1}{10} x_3^2x_6^5
+\frac{x_3 x_4^2x_6^3}{\sqrt{2}}
+\sqrt{3} x_3 x_4 x_5^2x_6
+\frac{x_3 x_5^2 x_6^4}{2 \sqrt{2}}
\\&
  +\frac{1}{4} x_4^4 x_6
 +\frac{3}{2} x_4^2 x_5^2x_6^2
 +\frac{1}{14} x_4^2 x_6^7
  +\frac{x_4 x_5^4}{\sqrt{6}}
+\frac{1}{2} \sqrt{\frac{3}{2}} x_4 x_5^2 x_6^5
\\&
  +\frac{1}{2} x_5^4 x_6^3
+\frac{1}{16} x_5^2 x_6^8
   +\frac{x_6^{13}}{1716}~.
 \end{split}\end{equation}

\begin{remark}
For $E_6$, Saito--Yano--Sekiguchi obtained the system of flat generators $y_m$ ($m=12,9,8,6,5,2$)
in \cite[(3.3)]{SaitoYanoSekiguchi}. They are  expressed in terms of the set of invariants 
$A,B,C,H,J,K$
obtained in \cite{Frame1951}.   If we identifiy the linear  coordinates $(u_1,\ldots,u_6)\in \mathbb{C}^6$ here 
and $(x_1,x_2,x_3,y_1,y_2,y_3)\in \mathbb{C}^6$ in \cite{Frame1951} as 
\begin{equation}\nonumber\begin{split}
&u_1\to \frac{x_1+x_2+x_3}{\sqrt{3}},\quad
u_2\to y_3,\quad u_3\to y_1,\quad u_4\to y_2,\\
&u_5\to \frac{-x_1-x_2+2x_3}{\sqrt{6}},\quad
u_6\to \frac{x_1-x_2}{\sqrt{2}},
\end{split}\end{equation}
then 
\begin{equation}\nonumber
\begin{split}
P_{12}&=360K+1000 HA^2+32C^2+192CA^3+4480B^2A+12A^6,\\
P_9&=168\sqrt{3}J+336\sqrt{3}BA^2,\quad 
P_8=80H+64CA+12A^4,\\
P_6&=24C+12A^3,\quad P_5=40\sqrt{3}B,\quad P_2=12A,
\end{split}
\end{equation}
and our set of good basic invariants and $y_m$'s are related as follows.
\begin{equation}\nonumber
\begin{split}
x_1^{(35)}&=6\sqrt{3}y_{12},\quad
x_2^{(35)}=12 y_9,\quad
x_3^{(35)}=3\sqrt{2}y_8,
\\
x_4^{(35)}&=\sqrt{\frac{3}{2}}y_6,\quad
x_5^{(35)}=2\sqrt{3}y_5,\quad
x_6^{(35)}=\frac{y_2}{2}.
\end{split}
\end{equation}
\end{remark}

\begin{remark}
In \cite[\S 3]{Abriani2009},
Abriani obtained a set of flat invariants $t_{m}$ $(m=12,9,8,6,5,2)$ and  the potential of $E_6$.
Our set of good basic invariants and his is related as follows.\footnote{
However, this is under the assumption that the invariant polynomials $P_{m}$ ($m=12,9,8,6,5,2$) here 
correspond to the counterparts  $u_m$ in \cite{Abriani2009} as
$$P_{12}\to u_{12},~~
P_{9}\to u_9,~~
P_{8}\to u_8,~~
P_{6}\to u_6,~~ 
P_5\to u_5,~~ 
P_2\to 12 u_2.$$ 
}
\begin{equation}\nonumber\begin{split}
x_1^{(35)}&=\frac{t_{12}}{20\sqrt{3}},\quad
x_2^{(35)}=\frac{t_9}{2\sqrt{3}},\quad
x_3^{(35)}=\frac{\sqrt{2}}{5}t_8, \\
x_4^{(35)}&=\frac{t_6}{8\sqrt{6}}, \quad
x_5^{(35)}=\frac{t_5}{20}, \quad
x_6^{(35)}=\frac{t_2}{2}.
\end{split}\end{equation}
Under this correspondence,
$200\mathcal{F}_{E_6}$  agrees with his potential.
\end{remark}
\subsection{An admissible triplet, good basic invariants 
and the potential vector field of $F_4=G_{28}$}
The finite Coxeter group $F_4=G_{28}$  has  rank $4$ and degrees $12,8,6,2$.
All reflections have order two.
We take the following simple roots:
\begin{equation}\begin{split}\nonumber
\alpha_1^{(28)}&=e_1-e_2,\quad
\alpha_2^{(28)}=e_2-e_3,\quad
\alpha_3^{(28)}=\sqrt{2}e_3,
\\
\alpha_4^{(28)}&=-\frac{1}{\sqrt{2}}(e_1+ e_2+ e_3+ e_{4})~.
\end{split}
\end{equation}
Then
$g_{(28)}=s(\alpha_1^{(28)})s(\alpha_2^{(28)})s(\alpha_3^{(28)})s(\alpha^{(28)}_4)
$ is a Coxeter element. Its
eigenvalues and eigenvectors are
\begin{equation}\nonumber
\begin{split}
\zeta:~~& q_1^{(28)}=\frac{1}{2\sqrt{3-\sqrt{3}}}
\begin{bmatrix}
-1+i\\ i (\sqrt{3}-1)\\ \sqrt{3}-1\\ -1-i
\end{bmatrix},
\quad 
\zeta^{5}:~~
q_2^{(28)}=\frac{1}{2\sqrt{3+\sqrt{3}}}
\begin{bmatrix}
-1+i\\ -i (1+\sqrt{3})\\ -1-\sqrt{3}\\ -1-i
\end{bmatrix},
\\
\zeta^{7}:~~&
q_3^{(28)}=\frac{1}{2\sqrt{3+\sqrt{3}}}
\begin{bmatrix}
1-i\\-1-\sqrt{3}\\ -i (1+\sqrt{3})\\ -1-i
\end{bmatrix},
\quad
\zeta^{11}:~~
q_4^{(28)}=
\frac{1}{2\sqrt{3-\sqrt{3}}}\begin{bmatrix}
1-i\\ \sqrt{3}-1\\ i (\sqrt{3}-1)\\ -1-i
\end{bmatrix}~.
\end{split}
\end{equation}
Then $(g_{(28)},\zeta, q_1^{(28)})$ is an admissible triplet of $G_{28}=F_4$.
The set of basic invariants which is good with respect to 
$(g_{(28)},\zeta,q_1^{(28)})$ and which is compatible at $q_1^{(28)}$ with 
the $(g_{(28)},\zeta)$-graded coordinates associated with $q_{\alpha}^{(28)}$'s 
is given as follows.
\begin{equation}
\begin{split}
x_1^{(28)}&=\frac{1}{10\sqrt{3}}
\left(
I_{12}+\frac{2959}{20736} I_6 I_2^3-\frac{77}{288}I_6^2-\frac{209}{576}I_8 I_2^2
-\frac{737}{331776} I_2^6
\right),
\\
x_2^{(28)}&=\frac{3}{20\sqrt{2}}\left(
I_8-\frac{7}{12}I_6I_2+\frac{385}{31104}I_2^4
\right),
\\
x_3^{(28)}&=-\frac{i}{4\sqrt{6}}\left(I_6-\frac{5}{144}I_2^3\right),
\\
x_4^{(28)}&=-\frac{i}{12}I_2~.
\end{split}
\end{equation}
These agree with the one given in \cite[(4.3)]{SaitoYanoSekiguchi}
\footnote{As pointed out in \cite{FeiginVeselov}, there are typos in the generators  given
in \cite[(4.3)]{SaitoYanoSekiguchi}.
The invariants $y_6$ and  $y_{12}$ there should be read as
$$y_6=-\frac{1}{8}I_6+\frac{15}{16}\left(\frac{I_2}{6}\right)^3,\qquad
y_{12}=-\frac{1}{60}I_{12}+\cdots + \frac{2211}{1280} \left(\frac{I_2}{6}\right)^6.$$}
up to constant multiples.
Here \cite{Mehta1988} (see also \cite[\S 4]{Iwasaki2002})
\begin{equation}\nonumber
I_{2k}=(8-2^{2k-1})S_{2k}+\sum_{i=1}^{k-1} \begin{pmatrix}2k\\2i\end{pmatrix}
S_{2i}S_{2k-2i}\quad (k=6,4,3,1),
\end{equation}
where
$$
S_m=u_1^m+u_2^m+u_3^m+u_4^m~.
$$

The potential vector field is given as follows.
 For the sake of simplicity, we omit the superscript from $x_{\alpha}^{(28)}$ here.
$$
\mathcal{G}_{\gamma}^{(28)}=\frac{\partial }{\partial x_{5-\gamma}}\mathcal{F}_{F_4}\quad
(1\leq \gamma\leq 4),
$$
where
\begin{equation}\label{pot28}
\begin{split}
\mathcal{F}_{F_4}&=
\frac{1}{2} x_1^2 x_4+x_1 x_2 x_3
+\frac{x_2^3 x_4}{3\sqrt{2}}
+\frac{1}{10} x_2^2 x_4^5
+\frac{x_2 x_3^2 x_4^3}{\sqrt{2}}
+\frac{1}{4} x_3^4 x_4
+\frac{1}{14} x_3^2 x_4^7
+\frac{x_4^{13}}{1716}~.
\end{split}
\end{equation}

\subsection{An admissible triplet, good basic invariants
and the potential vector field of $G_{25}$}
The duality group $G_{25}$ has  rank $3$ and degrees $12,9,6$.
Every reflection of $G_{25}$ has order three. 
The line system $\mathcal{L}_3$ of $G_{25}$ consists of $12$ lines.
Nine of them are the orbit of $\mathbb{C}(e_1+e_2+e_3)$
by $G(3,1,3)$. The remaining $3$  are the coordinate axis \cite[\S 8.5.3]{LehrerTaylor}.
Among these lines,
we choose the three lines spanned by:
\begin{equation}\nonumber
\alpha_1^{(25)}=\begin{bmatrix}0\\0\\1\end{bmatrix},\quad
\alpha_2^{(25)}=\begin{bmatrix}\omega\\\omega^2\\1\end{bmatrix},\quad
\alpha_3^{(25)}=\begin{bmatrix}1\\0\\0\end{bmatrix}~.
\end{equation}
Let
$g_{(25)}=s(\alpha_1^{(25)},\omega)
s(\alpha_2^{25},\omega)s(\alpha_3^{(25)},\omega).
$
Its eigenvalues and 
eigenvectors are
\begin{equation}
\begin{split}\nonumber
\zeta:~~ &q_1^{(25)}
=\frac{1}{\sqrt{6+2\sqrt{3}}}\begin{bmatrix}1\\1+\sqrt{3}\\1\end{bmatrix}
,\quad
\zeta^{4}=\omega:~~
q_2^{(25)}=\frac{1}{\sqrt{2}}
\left[\!\!\!\begin{array}{r}1\\0\\-1\end{array}\!\!\right],
\\
\zeta^{7}:~~&
q_3^{(25)}
=\frac{1}{\sqrt{6-2\sqrt{3}}}
\begin{bmatrix}1\\1-\sqrt{3}\\1\end{bmatrix}
.
\end{split}
\end{equation}
We take $(g_{(25)},\zeta, q_1^{(25)})$ as an admissible triplet of $G_{25}$
and the coordinates  $z_{(25)}$ dual to $q_{\alpha}^{(25)}$'s as $(g_{(25)},\zeta)$-graded coordinates.
The set of basic invariants which is good with respect to $(g_{(25)},\zeta,
q_1^{(25)})$ and which is compatible with $z_{(25)}$ at $q_1^{(25)}$ is as follows.
\begin{equation}\begin{split}\nonumber
x_1^{(25)}=\frac{8\sqrt{3}}{81}\left(C_{12}-\frac{5}{8}C_6^2\right) ,
\quad
x_2^{(25)}=\frac{32\sqrt{2}}{9} C_9,
\quad
x_3^{(25)}=-\frac{\sqrt{6}}{9} C_6~,
\end{split}
\end{equation}
where \cite[eq.(9)]{Maschke1889}
\begin{equation}\nonumber
\begin{split}
C_{12}&=(u_1^3+u_2^3+u_3^3) \big((u_1^3+u_2^3+u_3^3)^3+216u_1^3u_2^3u_3^3\big),
\\
C_9&=(u_1^3-u_2^3)(u_2^3-u_3^3)(u_3^3-u_1^3),
\\
C_6&=u_1^6+u_2^6+u_3^6-10(u_1^3u_3^3+u_2^3u_3^3+u_3^3u_1^3)~.
\end{split}
\end{equation}
The potential vector field is as follows.
As in the previous subsections, 
we omit the superscript from $x_{\alpha}^{(25)}$ in the equation below.
\begin{equation}\label{pot25}
\begin{split}
\mathcal{G}_1^{(25)}&
=\frac{x_1^2}{2}-\frac{1}{2} \sqrt{\frac{3}{2}} x_2^2 x_3
+\frac{x_3^4}{4},
\\
\mathcal{G}_2^{(25)}&=x_1 x_2-\frac{\sqrt{3}}{2}  x_2 x_3^2,
\\
\mathcal{G}_3^{(25)}&=x_1 x_3+\frac{x_2^2}{2 \sqrt{2}}~.
\end{split}
\end{equation}
These agree with the flat coordinates and the vector potential of $G_{25}$ obtained in 
\cite[\S 5.3]{Arsie-Lorenzoni2016}.

\subsection{An admissible triplet, good basic invariants
and the potential vector field of $G_{8}$}
The duality group $G_8$ has rank two and degrees $12,8$.
It is generated by \cite[\S 6.3]{LehrerTaylor}
\begin{equation}\nonumber
r_4=\begin{bmatrix}1&0\\0&i\end{bmatrix},\quad
r_4'=\frac{1}{2}\left[\!\!\!
\begin{array}{rr}1+i&-1+i\\-1+i&1+i\end{array}
\!\! \right].
\end{equation}
We take $g_{(8)}=r_4r_4'$. Its eigenvalues and  eigenvectors are
\begin{equation}\nonumber
\zeta:~~  q_1^{(8)}=\frac{1}{\sqrt{6+2\sqrt{3}}}
\begin{bmatrix}
1+\sqrt{3}\\-1-i
\end{bmatrix},
\quad
\zeta^{5}:~~ q_2^{(8)}=\frac{1}{\sqrt{6-2\sqrt{3}}}
\begin{bmatrix}
\sqrt{3}-1\\1+i
\end{bmatrix}~.
\end{equation}
Then $(g_{(8)},\zeta,q_1^{(8)})$ is an admissible triplet of $G_8$.
The graded coordinate system dual to $q_1^{(5)},q_2^{(5)}$ is denoted $z_{(5)}$.
The set of  basic invariants 
which is good with respect to $(g_{(8)},\zeta, q_1^{(8)})$
and which is compatible with $z_{(8)}$ at $q_1^{(8)}$ 
is given by
$$
x_1^{(8)}=\frac{\sqrt{3}}{16}t_O,\quad
 x_2^{(8)}=\frac{3}{8\sqrt{2}}h_O~,
$$
where \cite[\S 6.6]{LehrerTaylor}
\begin{equation}\label{tO-hO}
t_O=u_1^{12}-33u_1^8u_2^4-33u_1^4u_2^8+u_2^{12},
\quad
h_O=u_1^8+14u_1^4u_2^4+u_2^8~.
\end{equation}
The potential vector field is 
\begin{equation}\label{pot8}
\mathcal{G}_1^{(8)}=\frac{x_1^2}{2}+\frac{x_2^3}{3\sqrt{2}},\quad
\mathcal{G}_2^{(8)}=x_1x_2~.
\end{equation}
In the above, we omit the superscript from $x_{\alpha}^{(8)}$ for simplicity.
These results agree with \cite{Arsie-Lorenzoni2016} and
\cite[Table C7]{KMS2018}.

\subsection{An admissible triplet, good basic invariants 
and the potential vector field of $G_5$}
The duality group
$G_5$ has rank two and degrees $12,6$. It is generated by \cite[\S 6.2]{LehrerTaylor}
\begin{equation}\nonumber
r_1=\frac{\omega}{2}\left[\!\!\!
\begin{array}{rr}
-1-i&1-i\\
-1-i&-1+i
\end{array}\!\!\right]
~,\quad
r_2'=\frac{\omega}{2}
\left[\!\!\!
\begin{array}{rr}
-1+i&1-i\\
-1-i&-1-i
\end{array}\!\!\right]
~.
\end{equation}
We take $g_{(5)}=(r_2'r_1)^{-1}$. 
Eigenvalues  and  eigenvectors of $g_{(5)}$ are
$$
\zeta: ~~ q_1^{(5)}=
\frac{1}{\sqrt{2}}\begin{bmatrix}1\\1
\end{bmatrix}~,
\quad 
\zeta^7:~~ q_2^{(5)}=
\frac{1}{\sqrt{2}}
\left[\!\!\!\begin{array}{r}
-1\\1
\end{array}\!\!\!\right]~.
$$
Then $(g_{(5)},\zeta,q_1^{(5)})$ is an admissible triplet of $G_5$.
The graded coordinate system dual to $q_1^{(5)},q_2^{(5)}$ is denoted $z_{(5)}$.
The set of basic invariants which is good with respect to $(g_{(5)},\zeta,q_1^{(5)})$
and which is compatible with $z_{(5)}$ at $q_1^{(5)}$ is given as follows.
$$
x_1^{(5)}=-\frac{1}{12}(f_T^3-6i\sqrt{3}t_T^2),\quad
x_2^{(5)}=
-t_T
$$
where \cite[\S 6.6]{LehrerTaylor}
\begin{equation}\nonumber 
f_T=u_1^4+2i\sqrt{3}u_1^2u_2^2+u_2^4,\quad
t_T=u_1^5 u_2-u_1 u_2^5~.
\end{equation}
The potential vector field is 
\begin{equation}\label{pot5}
\mathcal{G}_1^{(5)}=\frac{x_1^2}{2}-\frac{x_2^4}{4},\quad
\mathcal{G}_2^{(5)}=x_1x_2~.
\end{equation}
In the above, we omit the superscript from $x_{\alpha}^{(5)}$ for simplicity.
These results agree with \cite{Arsie-Lorenzoni2016} and
\cite[Table C7]{KMS2018}.

\subsection{The $2$-reflection subquotient of 
$G_{35}=E_6$ is $G_{28}=F_4$}
Consider the reflection subquotient  of $G_{35}=E_6$ by $\delta=2$ with $2$-regular element $g_{(35)}^6$. Among the degrees of $G_{35}$, $d_1=12,d_3=8,d_4=6,d_6=2$ are divisible by $\delta=2$.
Therefore the $\zeta^6$-eigenspace $V(g_{(35)}^6,\zeta^6)$ 
 of $g_{(35)}^6$ 
 is spanned by
$q_1^{(35)},q_3^{(35)},q_4^{(35)},q_6^{(35)}$. 
Let us consider  the vector space isomorphism  $V(g_{(35)}^6,\zeta^6)\to \mathbb{C}^4$ given by
\begin{equation}\label{map35-28}
q_1^{(35)}\mapsto q_1^{(28)},\quad q_3^{(35)}\mapsto q_2^{(28)},\quad
q_4^{(35)}\mapsto q_3^{(28)},\quad q_6^{(35)}\mapsto q_4^{(28)}.
\end{equation}

Under this isomorphism, the reflection subquotient
$G_{35}(2)$ maps to $GL_4(\mathbb{C})$.
We show that this map gives an isomorphism between $G_{35}(2)$ and $G_{28}$.
Since the degrees of the both groups coincide, the orders are the same.
Therefore we only have to show that,
for each generator $s(\alpha_i^{(28)})$  ($1\leq i\leq 4$) of $G_{28}$, 
there exists an element of $G_{35}(2)$ which is mapped to $s(\alpha_i^{(28)})$.
As it turns out, 
\begin{equation}
\begin{split}\label{map35-28-2}
s(2e_3)&\mapsto s(\alpha_1^{(28)}),\\
s\left(e_2+e_4-\frac{1}{\sqrt{2}}e_5+\sqrt{\frac{3}{2}}e_6\right)
&\mapsto s(\alpha_2^{(28)}),
\\
s(2e_4)\circ s(2e_2)&\mapsto s(\alpha_3^{(28)}),\\
s(e_1-e_2-\sqrt{2}e_5)\circ
s\left(e_1-e_4+\frac{1}{\sqrt{2}}e_5+\sqrt{\frac{3}{2}}e_6\right)
&\mapsto s(\alpha_4^{(28)})~.
\end{split}
\end{equation}
Therefore $G_{35}(2)\cong G_{28}$.

Under the isomorphism \eqref{map35-28},
the good basic invariants of $G_{35}$ maps to those of $G_{28}$ as follows:
\begin{equation}\nonumber\begin{split}
&x_1^{(35)}\mapsto x_1^{(28)},\quad x_2^{(35)}\mapsto 0, \quad x_3^{(35)}\mapsto x_2^{(28)},\quad
\\
&x_4^{(35)}\mapsto x_3^{(28)},\quad x_5^{(35)}\mapsto 0,\quad  x_6^{(35)}\mapsto x_4^{(28)}.
\end{split}
\end{equation}
If we substitute these into the potential function $\mathcal{F}_{E_6}$, 
we obtain $\mathcal{F}_{F_4}$. Compare \eqref{pot35} and \eqref{pot28}.

\subsection{The $3$-reflection subquotient $G_{35}=E_6$ is $G_{25}$}
We consider the reflection subquotient of $G_{35}=E_6$ by $\delta=3$ with $3$-regular element $g_{(35)}^4$.
Among the degrees of $G_{35}=E_6$, $d_1=12,d_2=9,d_4=6$ are divisible by $\delta=3$.
Therefore the $\zeta^4$-eigenspace $V(g_{(35)}^4,\zeta^4)$ of $g_{(35)}^4$  is spanned by
$q_1^{(35)},q_2^{(35)},q_4^{(35)}$. 
Let us consider the vector space
 isomorphism  $V(g_{(35)}^4,\zeta^4)\to \mathbb{C}^3$ given by
\begin{equation}\label{map35-25}
q_1^{(35)}\mapsto e^{\frac{5\pi i}{4}}q_1^{(25)},\quad 
q_2^{(35)}\mapsto q_2^{(25)},\quad
q_4^{(35)}\mapsto e^{\frac{3\pi i}{4}}q_3^{(25)}.
\end{equation}

Under \eqref{map35-25},  we show that the reflection subquotient $G_{35}(3)$ is mapped to
$G_{25}\subset GL_3(\mathbb{C})$ isomorphically.
As in the case of the reduction from $G_{35}$ to $G_{28}$,
it is enough to show that there exists an element of $G_{35}$ which is mapped to 
each generator of $G_{25}$.  Indeed we have 
\begin{equation}\begin{split}
\nonumber 
s\left(
e_2-e_3+\frac{1}{\sqrt{2}}e_5+\sqrt{\frac{3}{2}}e_6
\right)\circ
s\left(
e_2+e_3-\frac{1}{\sqrt{2}}e_5-\sqrt{\frac{3}{2}}e_6
\right)
&\mapsto s (\alpha_1^{(25)},\omega),
\\
s(e_3-e_4-\sqrt{2}e_5)\circ
s(e_1-e_2-e_3+e_4)&\mapsto s(\alpha_2^{(25)},\omega ),
\\
s(2e_4)\circ s(e_3-e_4+\sqrt{2}e_5)&\mapsto 
s(\alpha_3^{(25)},\omega).
\end{split}
\end{equation}
Therefore $G_{35}(3)$ is isomorphic to $G_{25}$.

Notice that  $q_{\alpha}^{(35)}\mapsto a_{\alpha}q_{\beta}^{(25)}$ implies the correspondence
of the graded coordinates
$z_{\alpha}^{(35)}\mapsto a_{\alpha}^{-1}z_{\beta}^{(25)}$ ($\alpha=1,2,4$)
and that of good basic invariants
\begin{equation}\label{x-change}
x_{\alpha}^{(35)} \mapsto a_1^{-d_{\alpha}+1}a_{\alpha}^{-1}x_{\beta}^{(25)}~
(\alpha=1,2,4).
\end{equation}
See \cite[Lemma 4.4]{KM2023} for the reason of the factor in the right hand side.
Then, by Remark \ref{remark-pvf}, 
\begin{equation}\label{pvf-change}
\mathcal{G}_{\beta}^{(25)}=a_1^{d_1+d_{\alpha}-1}a_{\alpha}\mathcal{G}_{\alpha}^{(35)}.
\end{equation}

Therefore under the map \eqref{map35-25},  we have
\begin{equation}\nonumber\begin{split}
&x_1^{(35)}\mapsto  (e^{\frac{5\pi i}{4}})^{-12}\cdot x_1^{(25)}=-x_1^{(25)},\quad 
\\&x_2^{(35)}\mapsto (e^{\frac{5\pi i}{4}})^{-9+1}x_2^{(25)}=x_2^{(25)},  \quad 
\\&x_4^{(35)}\mapsto  (e^{\frac{5\pi i}{4}})^{-6+1}\cdot e^{-\frac{3\pi i}{4}}\cdot x_3^{(25)}=-x_3^{(25)},\\
&x_3^{(35)},x_5^{(35)},x_6^{(35)}\mapsto 0~.
\end{split}
\end{equation}
Comparing \eqref{pot35} and \eqref{pot25}.
we  see that
\begin{equation}\nonumber
\begin{split}
&(e^{\frac{5\pi i}{4}})^{12+12}\mathcal{G}^{(35)}_1=\mathcal{G}_1^{(35)}\to \mathcal{G}_1^{(25)},\\
&(e^{\frac{5\pi i}{4}})^{12+9-1}\mathcal{G}_2^{(35)}=-\mathcal{G}_2^{(35)}\to \mathcal{G}_2^{(25)},
\\
&(e^{\frac{5\pi i}{4}})^{12+6-1}\cdot e^{\frac{3\pi i}{4}}\mathcal{G}_4^{(35)}=\mathcal{G}_4^{(35)}\to \mathcal{G}_3^{(25)}~.
\end{split}
\end{equation}

\subsection{The $6$-reflection subquotient  of $G_{28}=F_4$ is  $G_5$}
We consider the reflection quotient of $G_{28}=F_4$ by $\delta=6$ with $6$-regular element $g_{(28)}^2$. 
Among the degrees of $G_{28}=F_4$, $d_1=12$ and $d_3=6$ are divisible by $\delta=6$.
Therefore the $\zeta^2$-eigenspace $V(g_{(28)}^2,\zeta^2)$  of $g_{(28)}^2$  is spanned by $q_1^{(28)},q_3^{(28)}$. 
Let us consider the vector space isomorphism
($V(g_{(28)}^2,\zeta^2)\to \mathbb{C}^2$ given by
\begin{equation}\label{map28-5}
q_1^{(28)}\mapsto e^{\frac{5\pi i}{4}}q_1^{(5)},\quad q_3^{(28)}\mapsto iq_2^{(5)}.
\end{equation}

Under the isomorphism, we show that $G_{28}(6)$ is mapped to $G_5\subset GL_2(\mathbb{C})$
isomorphically.
\footnote{Given the rank and the degrees, we have two candidates 
 $G_5$ and $G(6,1,2)$ for the $6$-reflection subquotient $G_{28}(6)$ of $G_{28}$.
We show that $G_{28}(6)\cong G_5$ by constructing the isomorphism explicitly.
The case of  $G_{25}(6)$ is similar.
}
As in the previous cases, we only have to show that there exists an element of
$G_{28}(6)$ which is mapped to each generator of $G_5$. This can be achieved by the following:
\begin{equation}\begin{split}\label{map28-5-2}
s(e_1-e_2)\circ s(e_1+e_3)
&\mapsto r_1,
\\
s(\sqrt{2}e_2)\circ 
s\left(\frac{1}{\sqrt{2}}(e_1-e_2+e_3+e_4)\right)
&\mapsto r_2'~.
\end{split}
\end{equation}
Thus $G_{28}(6)\cong G_5$.

Under the map \eqref{map28-5}, 
the good basic invariants of $G_{28}$ and $G_{5}$ correspond as follows.
\begin{equation}\nonumber\begin{split}
&x_1^{(28)}\mapsto  (e^{\frac{5\pi i}{4}})^{-12}\cdot x_1^{(5)}=-x_1^{(5)},\quad 
x_3^{(28)}\mapsto  (e^{\frac{5\pi i}{4}})^{-6+1} \cdot i ^{-1}\cdot x_2^{(5)}
 =e^{\frac{\pi i}{4}}\cdot x_2^{(5)},
\\&x_2^{(28)},x_4^{(28)}\mapsto 0~.
\end{split}
\end{equation}
Substituting  these into \eqref{pot28} and comparing \eqref{pot5},  we see that
\begin{equation}\nonumber
\begin{split}
&(e^{\frac{5\pi i}{4}})^{12+12}\mathcal{G}_1^{(28)}=\mathcal{G}_1^{(28)}\to \mathcal{G}_1^{(5)},\\
&
(e^{\frac{5\pi i}{4}})^{12+6-1}\cdot i\mathcal{G}_3^{(28)}=e^{\frac{7\pi i}{4}}\mathcal{G}_3^{(28)}
\to \mathcal{G}_1^{(5)}.
\end{split}
\end{equation}

\subsection{The $4$-reflection subquotient $G_{28}=F_4$ is $G_8$}
We consider the reflection quotient of $G_{28}$ by $\delta=4$ with $4$-regular element $g_{(28)}^3$.
Among the degrees of $G_{28}=F_4$, $d_1=12$ and $d_2=8$ are divisible by $\delta=4$.
Therefore the $\zeta^3$-eigenspace $V(g_{(28)}^3,\zeta^3)$  
 of $g_{(28)}^3$  is spanned by
$q_1^{(28)},q_2^{(28)}$. 
Let us consider the vector space isomorphism  $V(g_{(28)}^3,\zeta^3)\to \mathbb{C}^2$ given by
\begin{equation}\label{map28-8}
q_1^{(28)}\mapsto q_1^{(8)},\quad q_2^{(28)}\mapsto q_2^{(8)}.
\end{equation}

Under this isomorphism, there exists an element of $G_{28}(4)$ which is mapped to each generator of $G_8$ as follows:
\begin{equation}\begin{split}
\nonumber 
s(e_2-e_3)\circ s(\sqrt{2}e_3)
&\mapsto r_4,
\\
s(e_1+e_3)\circ 
s\left(-\frac{1}{\sqrt{2}}(e_1+e_2+e_3+e_4)\right)
&\mapsto r_4'~.
\end{split}
\end{equation}
Therefore $G_{28}(4)\cong G_8$.

Under the map \eqref{map28-8}, the good basic invariants of $G_{28}$ and $G_8$ correspond
as follows.
$$
x_1^{(28)}\mapsto x_1^{(8)},\quad x_2^{(28)}\mapsto x_2^{(8)}, \quad 
x_3^{(28)},x_4^{(28)}\mapsto 0.
$$
If we substitute these into
the potential vector field of $G_{28}=F_4$,
we obtain
that of $G_8$.
See \eqref{pot28}\eqref{pot8}.

\subsection{The $6$-reflection subquotient of $G_{25}$ is $G_5$}
Finally let us consider the reflection subquotient of $G_{25}$ by $\delta=6$ with the $6$-regular element $g_{(25)}^2$.
Among the degrees of $G_{25}$, $d_1=12$ and $d_3=6$ are divisible by $\delta=6$.
Therefore the $\zeta^2$-eigenspace  $V(g_{(25)}^2,\zeta^2)$ of $g_{(25)}^2$ is spanned by
the eigenvectors
$q_1^{(25)},q_3^{(25)}$.  
Let us consider  the isomorphism  $V(g_{(25)}^2,\zeta^2)\to \mathbb{C}^2$ given by
\begin{equation}\label{map25-5}
q_1^{(25)}\mapsto q_1^{(5)},\quad q_3^{(25)}\mapsto e^{\frac{\pi i}{4}}q_2^{(5)}.
\end{equation}
Under the isomorphism, the following elements of $G_{25}(6)$ are mapped to the generators of 
$G_5$ as follows:
\begin{equation}\begin{split}
\nonumber 
s\left(\begin{bmatrix}1\\ \omega^2\\ 1\end{bmatrix},\omega\right)&\mapsto (r_2')^2=(r_2')^{-1},
\quad
s\left(\begin{bmatrix}1\\ 0\\ 0\end{bmatrix},\omega\right)
\circ
s\left(\begin{bmatrix}0\\ 0\\ 1\end{bmatrix},\omega\right)
\mapsto r_1^2=r_1^{-1}~.
\end{split}
\end{equation}
Thus $G_{25}(6)\cong G_5$.

Under the map \eqref{map25-5}, 
the good basic invariants of $G_{25}$ and $G_{5}$ correspond as follows.
$$
x_1^{(25)}\mapsto  x_1^{(5)},\quad 
x_2^{(25)}\mapsto 0,  \quad 
x_3^{(25)}\mapsto   e^{-\frac{\pi i}{4}}x_2^{(5)}.
$$
Substitute these into \eqref{pot25} and comparing with \eqref{pot5}, we see that
\begin{equation}\nonumber\begin{split}
&\mathcal{G}_1^{(25)}\to \mathcal{G}_1^{(5)},\quad
e^{\frac{\pi i}{4}}\mathcal{G}_3^{(25)}\to \mathcal{G}_2^{(5)}~.
\end{split}\end{equation}

\section{Sequence of reflection subquotients of $G_{31}$}
\label{sec:G31}
In this section,
$$
\zeta=e^{\frac{\pi i}{12}}~,
$$
and $t_O$, $h_O$ are the same as \eqref{tO-hO}.

\subsection{Admissible triplet and good basic invariants of $G_{31}$}
$G_{31}$ is a non-duality group of rank $4$ with degrees $24,20,12,8$.
Among $60$ roots of $G_{31}$,
$28$ are those of $G(4,2,4)$:
$$
e_{\alpha}-i^k e_{\beta}~~(0\leq \alpha<\beta\leq 4,0\leq k<4),\quad 
e_{\alpha}~~(1\leq \alpha\leq 4)~,
$$
where $e_1,\ldots, e_4$ is the standard basic of $\mathbb{C}^4$.
The remaining $32$ roots are the $G(4,2,4)$-orbit of 
$e_1+e_2+e_3+e_4$.
Reflections of $G_{31}$ are the reflections of order $2$ with these roots. See \cite[\S 6.2, The line system $\mathcal{O}_4$]{LehrerTaylor}.
We take the following five roots
\begin{equation}
\begin{split}
\alpha^{(31)}_1=\begin{bmatrix}0\\0\\0\\1\end{bmatrix},~~
\alpha^{(31)}_2=\left[\!\!\begin{array}{r}1\\-1\\0\\0\end{array}\!\!\right],~~
\alpha^{(31)}_3=\begin{bmatrix}1\\0\\0\\1\end{bmatrix},~~
\alpha^{(31)}_4=\left[\!\!\begin{array}{r}1\\i\\1\\-i\end{array}\!\!\right],~~
\alpha^{(31)}_5=\begin{bmatrix}1\\1\\1\\1\end{bmatrix}~.
\end{split}
\end{equation}
We put
$$
g_{(31)}=s(\alpha^{(31)}_1)\circ s(\alpha^{(31)}_2)\circ s(\alpha^{(31)}_3)\circ s(\alpha^{(31)}_4)\circ s(\alpha^{(31)}_5)~.
$$
Its eigenvalues and eigenvectors are
\begin{equation}\nonumber
\begin{split}
\zeta&\text{ : }q_1^{(31)}=\frac{1}{\sqrt{2(2-\sqrt{2})(3+\sqrt{3})}}
\begin{bmatrix}
\frac{(\sqrt{2}-1)(\sqrt{3}+1)(1+i)}{2}\\
-\sqrt{2}+1\\
-\frac{(\sqrt{3}+1)(1+i)}{2}\\1\end{bmatrix},
\\
\zeta^5&\text{ : }q_2^{(31)}=\frac{1}{\sqrt{2(2+\sqrt{2})(3-\sqrt{3})}}
\begin{bmatrix}
\frac{(\sqrt{2}+1)(\sqrt{3}-1)(1+i)}{2}\\
\sqrt{2}+1\\
\frac{(\sqrt{3}-1)(1+i)}{2}\\1\end{bmatrix},
\\
\zeta^{13}&\text{ : }q_3^{(31)}=\frac{1}{\sqrt{2(2+\sqrt{2})(3+\sqrt{3})}}
\begin{bmatrix}
-\frac{(\sqrt{2}+1)(\sqrt{3}+1)(1+i)}{2}\\
\sqrt{2}+1\\
-\frac{(\sqrt{3}+1)(1+i)}{2}\\1\end{bmatrix},
\\
\zeta^{17}&\text{ : }q_4^{(31)}=\frac{1}{\sqrt{2(2-\sqrt{2})(3-\sqrt{3})}}
\begin{bmatrix}
-\frac{(\sqrt{2}-1)(\sqrt{3}-1)(1+i)}{2}\\
-\sqrt{2}+1\\
\frac{(\sqrt{3}+1)(1+i)}{2}\\1\end{bmatrix}.
\end{split}
\end{equation}
Then $(g_{(31)},\zeta,q_1^{(31)})$ is an admissible triplet of $G_{31}$. 
The coordinate system $z^{(31)}$ dual to $q_1^{(31)},\ldots,q_4^{(31)}$ is a $(g_{(31)},\zeta)$-graded coordinate system.
The set of basic invariants of $G_{31}$ which is good  with respect to  $(g_{(31)},\zeta,q_1^{(31)})$
and which is compatible with $z^{(31)}$ at $q_1^{(31)}$ is as follows.
\begin{equation}\label{good31}
\begin{split}
x_1^{(31)}&=\frac{1}{6}\left(F_{24}-\frac{7}{12}F_{12}^2+\frac{5}{16}F_8^3\right),
\\
x_2^{(31)}&=\sqrt{\frac{3}{2}}\left(
F_{20}-\frac{11}{12}F_{12}F_8
\right),
\\
x_3^{(31)}&=-\frac{1}{2\sqrt{3}}F_{12},\\
x_4^{(31)}&=-\frac{1}{2\sqrt{2}}F_{8}.
\end{split}
\end{equation}
Here  $F_{24},F_{20},F_{12},F_8$ are those defined in \cite[eqs.(7)(11)(12)]{Maschke1887}:
\begin{equation}\begin{split}
F_{24}&=\frac{1}{4}\Phi_1\Phi_2\Phi_3\Phi_4 \Phi_5\Phi_6,
\\
F_{20}&=\frac{1}{12}\sum_{1\leq i<j<k<l<\leq 6} \Phi_i\Phi_j \Phi_k \Phi_l \Phi_m~,
\\
F_{12}&=-\frac{1}{4}\sum_{1\leq i<j<k\leq 6}\Phi_i\Phi_j\Phi_k,\\
F_8&=-\frac{1}{6}\sum_{1\leq i<j\leq  6}\Phi_i\Phi_j~,
\end{split}
\end{equation}
where
\begin{equation}\nonumber
\begin{split}
\Phi_1&=u_1^4+u_2^4+u_3^4+u_4^4-6(
u_1^2u_2^2+u_1^2u_3^2+u_1^2u_4^2+u_2^2u_3^2+u_2^2u_4^2+u_3^2u_4^2
),
\\
\Phi_2&=u_1^4+u_2^4+u_3^4+u_4^4+6(
-u_1^2u_2^2+u_1^2u_3^2+u_1^2u_4^2+u_2^2u_3^2+u_2^2u_4^2-u_3^2u_4^2
),
\\
\Phi_3&=u_1^4+u_2^4+u_3^4+u_4^4+6(
u_1^2u_2^2-u_1^2u_3^2+u_1^2u_4^2+u_2^2u_3^2-u_2^2u_4^2+u_3^2u_4^2
),
\\
\Phi_4&=u_1^4+u_2^4+u_3^4+u_4^4+6(
u_1^2u_2^2+u_1^2u_3^2-u_1^2u_4^2-u_2^2u_3^2+u_2^2u_4^2+u_3^2u_4^2
),
\\
\Phi_5&=-2(u_1^4+u_2^4+u_3^4+u_4^4)-24u_1u_2u_3u_4,\\
\Phi_6&=-2(u_1^4+u_2^4+u_3^4+u_4^4)+24u_1u_2u_3u_4.
\end{split}
\end{equation}

\begin{remark}
In the cases of duality groups,  
the potential vector field \eqref{eq:pvf} gives 
the structure constants  $C_{\alpha\beta}^{\gamma}=\frac{\partial^2 \mathcal{G}_{\gamma}}{\partial x_{\alpha}\partial x_{\beta}}$ ($1\leq \alpha,\beta,\gamma\leq n$) of a multiplication
on the tangent bundle of the orbit space,
and 
the multiplication is commutative and associative \cite{KMS2018}\cite{KM2023}.
Although $G_{31}$ is not a duality group,
we can compute \eqref{eq:pvf} and obtain the followings.
\begin{equation}\nonumber
\begin{split}
\mathcal{G}_1^{(31)}&=
\frac{x_1^2}{2}+\frac{35}{108}x_3^4-\frac{x_2^2x_4}{3\sqrt{2}}-\frac{x_2x_3x_4^2}{\sqrt{2}}
-\frac{8\sqrt{2}}{3}x_3^2x_4^3+\frac{11}{12}x_4^6,
\\
\mathcal{G}_2^{(31)}&=
x_1x_2-\frac{2}{3}x_2x_3^2+\frac{11}{9}x_3^3x_4+\frac{x_2x_4^3}{\sqrt{2}}-\frac{7}{2\sqrt{2}}x_3x_4^4,
\\
\mathcal{G}_3^{(31)}&=x_1x_3+\frac{x_3^3}{9}-\frac{x_2x_4^2}{\sqrt{2}}-\frac{3}{\sqrt{2}}x_3x_4^3,
\\
\mathcal{G}_4^{(31)}&=x_1x_4+\frac{x_2x_3}{3}+x_3^2x_4-\frac{x_4^4}{2\sqrt{2}}.
\end{split}
\end{equation}
Here we write $x_{\alpha}$ instead of $x_{\alpha}^{(35)}$.
However,  it turns out that the induced  multiplication is not associative.

\end{remark}
\subsection{Admissible triplet, good basic invariants, 
and the potential vector field of $G_9$}
$G_9$ is a duality group of rank $2$, with degrees $24,8$.
It is generated by reflections \cite[\S 6.3]{LehrerTaylor}
\begin{equation}\nonumber
r_3=\frac{1}{2}\left[\!\!\begin{array}{rr}1&-1\\-1&-1\end{array}\!\!\right],\quad
r_4=\begin{bmatrix}1&0\\0&i\end{bmatrix}~
\end{equation}
of orders $2$ and $4$.
We put $$
g_{(9)}=r_3r_4~.
$$
Its eigenvalues and eigenvectors are as follows.
\begin{equation}\nonumber
\zeta:~~q_1^{(9)}=\frac{1}{\sqrt{3+\sqrt{3}}}
\begin{bmatrix}-\frac{(\sqrt{3}+1)(1+i)}{2}\\1\end{bmatrix}
,\quad
\zeta^{17}:~~q_2^{(9)}=\frac{1}{\sqrt{3-\sqrt{3}}}
\begin{bmatrix}\frac{(\sqrt{3}-1)(1+i)}{2}\\1\end{bmatrix}
\end{equation}
The vector $q_1^{(9)}$ is a regular vector.
Therefore $(g_{(9)},\zeta,q_1^{(9)})$ is an admissible triplet of $G_{9}$ and 
the coordinate system  dual to the basis $q_1^{(9)},q_2^{(9)}$ is a $(g_{(9)},\zeta)$-graded coordinate system.
The set of basic invariants of $G_9$ which is good with respect to $(g_{(9)},\zeta,q_1^{(9)})$
and which is compatible with the $(g_{(9)},\zeta)$-graded coordinate system at $q_1^{(9)}$ is 
\begin{equation}\label{good9}
x_1^{(9)}=\frac{9}{128}\left(t_{O}^2-\frac{11}{16}h_{O}^3\right),\quad
x_2^{(9)}=-\frac{3}{8\sqrt{2}}h_{O}~.
\end{equation}
The potential vector field is
\begin{equation}\label{pvf9}
\mathcal{G}_1^{(9)}=\frac{x_1^2}{2}+\frac{11}{12}x_2^6,\quad 
\mathcal{G}_2^{(9)}=x_1x_2-\frac{x_2^4}{2\sqrt{2}}~.
\end{equation}
In the above,  we write $x_{\alpha}$ for $x_{\alpha}^{(9)}$ for simplicity.
These results agree with \cite{Arsie-Lorenzoni2016} and
\cite[Table C7]{KMS2018}.

\subsection{Admissible triplet, good basic invariants, 
and the potential vector field of $G_{10}$}

$G_{10}$ is a duality group of rank $2$, with degrees $24,12$.
It is generated by reflections \cite[\S 6.3]{LehrerTaylor}
$$r_1=\frac{\omega}{2}\left[\!\!\begin{array}{rr}
-1-i&1-i\\
-1-i&-1+i
\end{array}\!\!\right]
~,\quad
r_4'=\frac{1}{2}\left[\!\!\begin{array}{rr}
1+i&-1+i\\-1+i&1+i
\end{array}\!\!\right]
$$
of orders $3$ and $4$.
We put $$
g_{(10)}=(r_1r_4')^5~.
$$
Its eigenvalues and eigenvectors are as follows.
\begin{equation}\nonumber
\zeta:~~q_1^{(10)}=\frac{1}{\sqrt{2}}
\begin{bmatrix}\frac{(1+i)}{\sqrt{2}}\\1\end{bmatrix},
\quad
\zeta^{13}:~~q_2^{(10)}=\frac{1}{\sqrt{2}}
\begin{bmatrix}-\frac{(1+i)}{\sqrt{2}}\\1\end{bmatrix}
\end{equation}
The vector $q_1^{(10)}$ is a regular vector.
Therefore $(g_{(10)},\zeta,q_1^{(10)})$ is an admissible triplet of $G_{10}$ and 
the coordinate system  dual to the basis $q_1^{(10)},q_2^{(10)}$ is a $(g_{(10)},\zeta)$-graded coordinate system.
The set of basic invariants of $G_{10}$ which is good with respect to $(g_{(10)},\zeta,q_1^{(10)})$
and which is compatible with the $(g_{(10)},\zeta)$-graded coordinate system at $q_1^{(10)}$ is 
\begin{equation}\label{good10}
x_1^{(10)}=-\frac{8}{81}\left(h_{O}^3-\frac{7}{12}t_{O}^2\right),\quad
x_2^{(10)}=\frac{2}{9}t_{O}~.
\end{equation}
The potential vector field is
\begin{equation}\label{pvf10}
\mathcal{G}_1^{(10)}=\frac{x_1^2}{2}+\frac{35}{108}x_2^4
,\quad 
\mathcal{G}_2^{(10)}=x_1x_2-\frac{x_2^3}{9}~.
\end{equation}
Above, we omit the superscript from   $x_{\alpha}^{(10)}$ for simplicity.
These results agree with \cite{Arsie-Lorenzoni2016} and
\cite[Table C7]{KMS2018}.

\subsection{The $8$-reflection subquotient $G_{31}$ is $G_9$}
We consider the reflection subquotient of $G_{31}$ by $\delta=8$ with $8$-regular element
$g_{(31)}^3$.
Among the degrees of $G_{31}$, $d_1=24$ and $d_4=8$  are divisible by $8$.
Therefore the $\zeta^3$-eigenspace $V(g_{(31)}^3,\zeta^3)$ is spanned by $q_1^{(31)}$, $q_4^{(31)}$.
Let us consider the vector space isomorphism 
$V(g_{(31)}^3,\zeta^3)\to \mathbb{C}^2 $ given by
$$
q_1^{(31)}\mapsto q_1^{(9)},\quad q_4^{(31)}\mapsto q_2^{(9)}.
$$
Under this isomorphism, the reflection subquotient $G_{31}(8)$ maps into $GL_2(\mathbb{C})$
and we have
 \begin{equation}\label{map31-9}
\begin{split}
s(e_1-e_2+e_3+e_4)\circ s(e_1-e_4)&\mapsto r_3,\\
g_{(31)}&\mapsto r_3r_4~.
\end{split}
\end{equation}
This means that   there exists an element of $G_{31}(8)$ which is mapped to
 each generator of $G_9$ and hence that  $G_9$ is contained in the image of the map.
 Since the degrees and hence the orders of  $G_{31}(8)$ and $G_9$ are the same,
 the map gives an isomorphism between $G_{31}(8)\cong G_9$.

Under the isomorphism \eqref{map31-9}, the good basic invariants are mapped as follows:
$$
x_1^{(31)}\to x_1^{(9)},\quad
x_4^{(31)}\to x_2^{(9)},\quad 
x_2^{(31)},x_3^{(31)}\to 0.
$$

\subsection{The $12$-reflection subquotient of $G_{31}$ is $G_{10}$}
Finally consider the reflection subquotient of $G_{31}$ by $\delta=12$ with  the
$12$-regular element $g_{(31)}^2$.
The degrees of $G_{31}$ divisible by $\delta=12$ are $d_1=24$ and $d_3=12$.
Therefore the  $\zeta^2$-eigenspace of $V(g_{(31)}^2,\zeta^2)$ is spanned by $q_1^{(31)}$, $q_3^{(31)}$. 
Under the isomorphism of the vector spaces $V(g_{(31)}^2,\zeta^2)\to \mathbb{C}^2$ given by
$$
q_1^{(31)}\mapsto q_1^{(10)},\quad q_3^{(31)}\mapsto i q_2^{(10)},
$$
the reflection subquotient $G_{31}(12)$ maps onto $G_{10}$ since
\begin{equation}\label{map31-10}
\begin{split}
s(e_1-e_2+ie_3+ie_4)\circ s(e_1+ie_2+e_3-ie_4)&\mapsto r_1,\\
g_{(31)}^5&\mapsto  r_1r_4'~.
\end{split}
\end{equation}
Under the isomorphism \eqref{map31-10}, the good basic invariants are mapped as follows:
$$
x_1^{(31)}\to x_1^{(10)},\quad
x_3^{(31)}\to -ix_2^{(10)},\quad 
x_2^{(31)},x_4^{(31)}\to 0.
$$


\end{document}